\documentclass[11pt,a4paper,reqno]{amsart}

\usepackage[T1]{fontenc}
\usepackage[utf8]{inputenc}
\usepackage{lmodern}
\usepackage{microtype}
\usepackage[
  a4paper,
  textwidth=15.2cm,
  textheight=23.2cm,
  centering,
  headheight=14pt
]{geometry}
\usepackage{mathrsfs}
\usepackage{amsmath,amssymb,amsthm}

\usepackage[
  colorlinks=true,
  linkcolor=blue,
  citecolor=magenta,
  urlcolor=blue
]{hyperref}

\newcommand{\EE}{{\mathcal  E}}
\newcommand{\FF}{{\mathcal  F}}

\def\XXint#1#2#3{{\setbox0=\hbox{$#1{#2#3}{\int}$ }
\vcenter{\hbox{$#2#3$ }}\kern-.6\wd0}}

\newtheorem{theorem}{\bf Theorem}[section]
\newtheorem{proposition}[theorem]{\bf Proposition}
\newtheorem{lemma}[theorem]{\bf Lemma}
\newtheorem{corollary}[theorem]{\bf Corollary}

\theoremstyle{definition}
\newtheorem{definition}[theorem]{Definition}

\newtheorem{remark}[theorem]{Remark}

\numberwithin{equation}{section}

\begin{document}

\title[Semilinear equations for stable L\'evy operators]
{Isolated Singularities and Measure Data Problems\\
for Semilinear Equations Driven by Stable L\'evy Operators}

\author[K. Dunst]{Kamil Dunst}
\author[T. Klimsiak]{Tomasz Klimsiak}

\address{
Faculty of Mathematics and Computer Science,
Nicolaus Copernicus University in Toru\'n,
Chopina 12/18,
87-100 Toru\'n,
Poland
}

\email{kdunst@doktorant.umk.pl}
\email{tomas@mat.umk.pl}



\begin{abstract}
We investigate positive solutions of semilinear equations driven by uniformly elliptic strictly $2s$-stable 
L\'evy  operators, where $s\in (0,1)$.
We first prove that every positive distributional solution of $-Lu=u^p$ in a punctured domain $D\setminus\{0\}$ satisfies
 $-Lu=u^p+k\delta_0$ in $D$ for some $k\ge0$, and that necessarily $k=0$ whenever $p\ge d/(d-2s)$. 
We then study the corresponding Dirichlet problem in which the Dirac mass is replaced by a bounded positive measure, 
and establish the existence of a critical parameter $k_\mu$: below this threshold minimal positive solutions exist, whereas above it the problem admits no solution. 
In the symmetric case, we further prove multiplicity below the threshold, as well as existence and uniqueness at the threshold itself.
\end{abstract}

\maketitle

\section{Introduction}
\label{sec1}

Semilinear equations with isolated singularities and measure-valued
sources provide a natural setting in which nonlinear effects, the
singular behaviour of Green functions, and critical integrability
interact. While these problems have been extensively studied for the
Laplacian and the fractional Laplacian, substantially less is known for
anisotropic and non-symmetric stable generators.

Let $D\subset\mathbb R^d$ be a bounded open set, with $d\ge 1$, and assume that $0\in D$. 
The aim of the present paper is to study, for $p>1$, the existence and qualitative properties of positive solutions to two interconnected problems:
the isolated-singularity problem
\begin{equation}
\label{eq1.0}
-Lu=u^p\quad\text{in }\mathscr D'(D\setminus\{0\}),
\end{equation}
and the parameter-dependent exterior Dirichlet problem
\begin{equation}
\label{eq1.1}
\begin{cases}
-Lu=u^p+k\mu,\quad &\text{in}\quad D, \\
u=0,\quad &\text{in}\quad\mathbb{R}^d\setminus D,
\end{cases}
\end{equation}
where $k>0$ and $\mu$ is a non-zero bounded  positive Borel measure on $D$.
The operator $L$ is the generator of a strictly $2s$-stable L\'evy process, where $s\in(0,1)$ and $2s<d$.
On $C_c^2(\mathbb R^d)$ it has the form
\begin{equation}
\label{L}
Lu(x)=b\cdot\nabla u(x)+\int_{\mathbb R^d}\bigl(u(x+y)-u(x)-y\cdot\nabla u(x)\mathbf 1_{B(0,1)}(y)\bigr)\,\nu(dy),
\end{equation}
where $b\in\mathbb R^d$ and $\nu$ is a L\'evy measure, not necessarily symmetric. Strict $2s$-stability means that
\begin{equation}
\label{scaling_def}
Lu_a(x)=a^{2s}Lu(ax),
\qquad a>0,\quad u\in C_c^2(\mathbb R^d),
\end{equation}
with $u_a(x):=u(ax)$. In particular, the L\'evy measure admits the polar representation
\begin{equation}
\label{eq.specme}
\nu(B)=\int_{\mathbb S^{d-1}}\lambda(d\theta)\int_0^\infty \mathbf 1_B(r\theta)\frac{dr}{r^{1+2s}},
\qquad B\in\mathcal B(\mathbb R^d),
\end{equation}
for a finite positive Borel measure $\lambda$ on $\mathbb S^{d-1}$, called the spectral measure; 
see \cite[Theorem~14.3]{Sato}. We impose the uniform non-degeneracy condition
\begin{equation}
\label{ellipticity}
\inf_{\xi\in\mathbb S^{d-1}}
\int_{\mathbb S^{d-1}}|\xi\cdot\theta|^{2s}\,\lambda(d\theta)\ge \Lambda
\end{equation}
for some $\Lambda>0$. Since $\lambda$ is finite, this is equivalent to two-sided comparability of the modulus 
of the symbol of $-L$ with $|\xi|^{2s}$. When $s\in(0,1/2)$, we additionally assume
\[
\operatorname{Cone}\bigl(\operatorname{Conv}(\operatorname{supp}\lambda)\bigr)=\mathbb R^d.
\]
This excludes one-sided processes and guarantees the strict positivity of the full-space Green function. 
In the symmetric case, it follows automatically from \eqref{ellipticity}.
Our results show that the principal removability and threshold phenomena persist under these broad structural assumptions.

\subsection{Main results}

The first part of the paper is devoted to positive distributional solutions of \eqref{eq1.0} and belongs to 
the broader theory of isolated singularities for semilinear equations. A preliminary difficulty is specific to the 
nonlocal setting: local integrability of \(u\) in \(D\setminus\{0\}\) does not by itself ensure that \(Lu\) is well defined as a 
distribution there, since the action of \(L\) also depends on the values of \(u\) outside the support of the test function. 
To accommodate this nonlocal dependence, Definition~\ref{def.space1} introduces, for every bounded open set 
\(V\subset\mathbb R^d\), the tail space \(\mathscr T_\nu(V)\) associated with the L\'evy measure \(\nu\).
This space satisfies
\[
L^1(\mathbb R^d)\subset \mathscr T_\nu(V)
\subset L^1_{\mathrm{loc}}(V)
\]
and is chosen so that, for every $u\in\mathscr T_\nu(V)$, the distribution $Lu\in\mathscr D'(V)$ is well defined.
We prove in Theorem~\ref{th.main1} that every positive solution
\[
u\in \mathscr T_\nu(D\setminus\{0\})\cap L^p_{\mathrm{loc}}(D\setminus\{0\})
\]
of \eqref{eq1.0} automatically belongs to
\[
u\in \mathscr T_\nu(D)\cap L^p_{\mathrm{loc}}(D)
\]
and satisfies
\begin{equation}
\label{eq-firstpl}
-Lu=u^p+k\delta_0\quad\text{in }\mathscr D'(D)
\end{equation}
for a  constant $k\ge0$. Thus the singularity at the origin can contribute only a non-negative multiple of the Dirac mass. Moreover,
\[
p\ge \frac{d}{d-2s}
\quad\Longrightarrow\quad
k=0.
\]
The exponent $d/(d-2s)$ is therefore the removability threshold for the singular measure generated by a positive distributional solution.
The proof also yields a Riesz-type decomposition. More precisely, Corollary~\ref{rem.ineq} shows that, almost everywhere in $D$,
\[
u(x)=h(x)+\int_DG_D(x,z)u^p(z)\,dz+kG_D(x,0)+\int_{D^c}u(z)
\left(
\int_{z-D}G_D(x,z-\xi)\,\nu(d\xi)
\right)dz,
\]
where $h$ is harmonic with respect to the Dirichlet realization of $L$ on $D$, and  $G_D$ is the Green function for $L$ and $D$. 
As a corollary (see Corollary \ref{col.clhol}), every classical   solution to the exterior Dirichlet   problem 
\begin{equation}
\label{eq1.1class}
\begin{cases}
-Lu=u^p,\quad &\text{in}\quad D\setminus\{0\}, \\
u=0,\quad &\text{in}\quad\mathbb{R}^d\setminus D,
\end{cases}
\end{equation}
belongs to  $L^p(D)$, and
\begin{equation}
\label{eq.repfor23intr}
u(x)=\int_DG_D(x,y)u^p(y)\,dy+kG_D(x,0),\quad x\in D.
\end{equation} 
In the subcritical range $p\in(1,\frac{d}{d-2s})$, one of the following alternatives holds.
If $k=0$ in \eqref{eq.repfor23intr}, then $u$ admits a locally H\"older continuous extension to $D$. If $k>0$  in \eqref{eq.repfor23intr}, 
then 
\[
kG_D(x,0)\le u(x)\le c kG_D(x,0),\quad x \in B_\varepsilon(0)
\]
for some  $c>1$ and $\varepsilon>0$.
In particular, in the present generality one should not expect the isotropic asymptotic behaviour
\[
u(x)\sim |x|^{2s-d}
\qquad\text{as }x\to0.
\]
Rather, the singular profile is governed by $G_D(x,0)$, whose behaviour may depend strongly on the direction of 
approach and on the spectral measure of the underlying stable process.
A finer directional analysis of this singular behaviour lies beyond the
scope of the present paper.
We also note a further phenomenon that does not arise for the fractional Laplacian. If the function 
$x\mapsto G_D(x,0)$ fails to be locally bounded in $D\setminus{\{0\}}$, then no classical solution in $D\setminus\{0\}$
can satisfy \eqref{eq.repfor23intr} with $k>0$. 

These results naturally lead to the solvability problem for \eqref{eq1.1}, where $p\in\left(1,\frac{d}{d-2s}\right)$
and $\mu$ is a non-zero bounded positive Borel measure on $D$. 
As the basic definition of a solution to problem \eqref{eq1.1}, we adopt the following integral formulation: 
a function $u\in L^p(D)$ is a solution if, for a.e. $x\in D$,
\begin{equation}
\label{eq.for.int.int}
u(x)=\int_D G_D(x,y)u^p(y)\,dy+k\int_D G_D(x,y)\,\mu(dy).
\end{equation}
No regularity assumption is imposed on the boundary of \(D\). We require only that the Dirichlet Green function \(G_D\) 
be strictly positive; this condition holds, for instance, when \(D\) is convex, as observed in Remark~\ref{rem.any1}. A general and 
readily verifiable sufficient criterion is established in Proposition~\ref{prop.com1} of the appendix.
We prove (see Theorem \ref{k-star-sol}, Proposition \ref{prop.uncrit}) that there exists a critical value $k_\mu\in(0,\infty)$ such that:
\begin{itemize}
\item for every $k\in(0,k_\mu)$, problem \eqref{eq1.1} admits a minimal positive solution $u_k$;
\item for $k=k_\mu$, problem \eqref{eq1.1} admits at most one positive solution;
\item for every $k>k_\mu$, problem \eqref{eq1.1} admits no positive solution.
\end{itemize}
If, in addition, $L$ is symmetric, then the minimal solution is stable (see Theorem \ref{stability}) 
and, for every $k\in(0,k_\mu)$, there exists at least one further positive solution
$v_k>u_k$ (see Theorem \ref{mountain-pass-sol}).  Moreover, a solution exists at $k=k_\mu$ 
and is therefore unique (see Theorem \ref{th.critex1}).

\subsection{Novelty and main difficulties}
\label{subsec-main-difficulties}

For the Laplacian, isolated singularities of positive solutions, removability, and the emergence of a 
Dirac mass in the distributional equation were developed in the classical works of Lions \cite{lit:Lions1980}, 
Brezis--V\'eron \cite{lit:BrezisVeron1980}, Gidas--Spruck \cite{lit:GidasSpruck1981}, 
Caffarelli--Gidas--Spruck \cite{lit:CaffarelliGidasSpruck1989}, 
and Baras--Pierre \cite{lit:BarasPierre1984}. Related existence, non-existence, and multiplicity questions for equations 
with singular or measure-valued data were studied, among others, in \cite{lit:AmannQuittner1998,lit:FerreroSaccon2006,lit:NaitoSato2007}. 
For more general local uniformly elliptic operators, analogous problems have been treated through capacities, 
Green potentials, and boundary potential theory; see, for instance, \cite{lit:HirataOno2014,lit:Veron1996,lit:MarcusVeron2014}.

In the nonlocal setting,  the best understood  model is $L=-(-\Delta)^s$. 
The existing analysis of isolated singularities for semilinear equations
driven by the fractional Laplacian follows, broadly speaking, two main
approaches. The first is based on the Caffarelli--Silvestre extension, which
realizes $(-\Delta)^s$ as a Dirichlet-to-Neumann operator for a local
degenerate elliptic equation in one additional variable. This representation
makes available local PDE techniques; see, for
instance,
\cite{lit:CaffarelliJinSireXiong2014,lit:DeLaTorreGonzalez2018,
lit:YangZou2021}.
In its standard form, however, this method is specific to the fractional
Laplacian and closely related symmetric operators. It does not yield a
suitable local extension problem for the general anisotropic and possibly
non-symmetric stable generators considered here.

The second approach works directly with Green operators, distributional
methods, maximum principles, and nonlocal potential theory. It has been used
both for semilinear nonlocal equations with measure data or source terms
\cite{lit:ChenVeron2014,GV,lit:HuynhNguyen2023,KR:CM}
and for equations with isolated singularities
\cite{lit:ChenQuaas2018,LiLiuWuXu2020,LiWuXu2018}.
Among the contributions belonging to this approach, the work of Chen and Quaas \cite{lit:ChenQuaas2018}
is most closely related to the two problems considered here. It combines a B\^ocher-type theorem for a semilinear equation
in a punctured domain with a critical-parameter problem involving a Dirac source, 
in the spirit of the pioneering work of Lions \cite{lit:Lions1980}. 
Although the Green-operator approach avoids the extension technique, the existing arguments continue to rely substantially 
on analytic properties specific to the fractional Laplacian, which sharply restricts their applicability to more general stable generators.
More specifically, the following properties, or closely related variants thereof, are
extensively exploited.
\begin{enumerate}
\item[\rm (i)]
\emph{The explicit formula for full-space Green function  and its smoothness off the diagonal.}
\item[\rm (ii)]
\emph{Sharp two-sided pointwise
estimates for the Dirichlet Green function.}
In sufficiently regular domains, one has an estimate of the form
\begin{equation}
\label{green.dircom}
G_D(x,y)
\asymp
G(x,y)
\left(
1\wedge\frac{\delta_D(x)}{|x-y|}
\right)^s
\left(
1\wedge\frac{\delta_D(y)}{|x-y|}
\right)^s,
\end{equation}
where $\delta_D(x):=\operatorname{dist}(x,D^c)$
(see \cite{ChenHuZhao2025,DipierroRosOtonSerraValdinoci2022}).  
\item[\rm (iii)]
\emph{Pointwise estimates for spatial derivatives of the Dirichlet Green
function.}
In particular, the decomposition
\[
G_D(x,y)=G(x,y)-g(x,y)
\]
can be combined with estimates for the derivatives of the regular part,
such as
\[
\left|\partial_{x_i}g(x,0)\right|
\leq C\delta_D(x)^{s-1}
\]
(see \cite{RosOtonSerra2016,Sztonyk2010v1}).
\item[\rm (iv)]
\emph{Smoothness of the density of the jump measure  away from the origin.}
For the fractional Laplacian its density belongs to $C^\infty(\mathbb R^d\setminus\{0\})$, and, 
 for every $\varepsilon>0$, all its first derivatives are
bounded on $\mathbb R^d\setminus B(0,\varepsilon)$.
\item[\rm (v)]
\emph{Harnack inequalities and related comparison principles} (see \cite{BogdanSztonyk2007,DipierroRosOtonSerraValdinoci2022,RosOtonSerra2016}).
\item[\rm (vi)]
\emph{Schauder estimates and boundary regularity results in the form required by the classical-solution approach} (see \cite{RosOtonSerra2016,DipierroRosOtonSerraValdinoci2022}).
\item[\rm (vii)]
\emph{Quasi-metric structure of the Green kernel.}
More precisely, in many of the relevant settings either
\[
d_G(x,y):=\frac{1}{G_D(x,y)}
\]
satisfies a quasi-triangle inequality, or $G_D$ becomes a quasi-metric
kernel after multiplication by a suitable boundary modifier (see \cite{GV,lit:HuynhNguyen2023}).
\end{enumerate}
In particular, Chen and Quaas \cite{lit:ChenQuaas2018} make essential use of properties (i)--(iv). Schauder estimates and boundary regularity of $D$ 
also play an important role there, since their analysis is carried out for classical solutions in $D\setminus{\{0\}}$
that vanish on $D^c$. This creates an additional difficulty in our setting, where we consider 
distributional solutions and 
impose only a tail-integrability condition on the exterior values of the solution, encoded by the spaces
$\mathscr T_\nu(D)$.

None of the properties listed above is available  under the general assumptions imposed 
in the present paper. Accordingly, the novelty of our contribution lies not only in the results themselves 
but also - and perhaps primarily - in the methods developed to establish them. Our analysis combines analytic and 
probabilistic potential theory with operator-theoretic techniques and is formulated systematically in terms of the killed resolvent. 
The guiding principle is to derive the properties of
the Green operator that are needed in the nonlinear analysis from the limited structural information available for the L\'evy operator, 
namely its scaling property and the two-sided comparability of its Fourier symbol with
$|\xi|^{2s}$.

Several stages of the analysis require a substantial departure from the methods employed in the existing literature. 
This is particularly apparent in the treatment of the isolated-singularity problem, where the argument is based on Dynkin's 
formula and the Riesz decomposition theorem and therefore draws on fundamental results of probabilistic potential theory.
A second methodological departure arises in the construction of solutions to the measure-data problem. 
In the fractional-Laplacian setting, monotone iteration is commonly controlled by an explicit supersolution derived from a 
pointwise nonlinear potential estimate such as 
\[ 
G_D\bigl[(G_D\mu)^p\bigr]\leq C G_D\mu. 
\] 
Under the present assumptions, such an estimate cannot be obtained directly in the absence of pointwise 
bounds for the Green function. Instead, we combine Deny's theorem (see Theorem \ref{th-denys}) with probabilistic and operator-theoretic arguments 
to establish compactness of the nonlinear Green operator. Schauder's fixed-point theorem then yields a solution 
for a sufficiently small parameter, which subsequently serves as a barrier in the monotone iteration scheme.
The proofs of other results retain the general architecture of earlier
arguments but require substantial modifications. 

The section devoted to the critical parameter (Section \ref{sec9}) addresses the existence of a solution for \(k=k_\mu\). 
Since the minimal solutions \(u_k\) increase monotonically as \(k\uparrow k_\mu\), the monotone convergence theorem permits 
passage to the limit in the integral formulation \eqref{eq.for.int.int}. The essential difficulty is to prove that the resulting extremal function 
\[
u^*:=\sup_{k<k_\mu}u_k 
\]
belongs to \(L^p(D)\), as required by the definition of a solution. 
Global \(L^p(D)\)-integrability of the extremal solution is not established explicitly even 
in the fractional-Laplacian analysis of Chen and Quaas. 
To overcome this difficulty, we adapt, among other ingredients, several ideas introduced by Ferrero and Saccon \cite{lit:FerreroSaccon2006} 
in their study of elliptic equations with source nonlinearities and measure data.

\subsection{Organization of the paper}

Section~\ref{sec2} introduces the notation and recalls the analytic and probabilistic background on strictly 
stable L\'evy generators, their killed processes, and Green functions. Section~\ref{sec3} develops the properties of 
the Green operator used throughout the paper. In Section~\ref{sec4}, we recall Deny's theorem in the required form and 
prove compactness results for Green-potential operators. Section~\ref{sec5} introduces the distributional and classical 
notions of solution adapted to the nonlocal setting. Section~\ref{sec6} is devoted to isolated singularities and proves 
the removability and asymptotic results described above. Section~\ref{sec7} studies the Dirichlet problem with measure data, 
constructs the minimal branch and the critical parameter, and proves that there is at most one solution at the threshold. 
In Section~\ref{sec8}, under the symmetry assumption, we prove stability of the minimal branch and multiplicity 
below the threshold. Section~\ref{sec9} establishes existence, and hence uniqueness, of the critical solution in the symmetric case. 
The appendices collect supplementary results concerning the Fourier symbol, fine topology, and positivity and comparison 
properties of Green functions.

\section{Preliminaries}
\label{sec2}
Let $\ell^d$ denote the Lebesgue measure on $\mathbb{R}^d$.  Throughout the paper $D$ denotes an open bounded subset of $\mathbb R^d$.
We write $\mathcal{B}(\mathbb{R}^d)$ for the Borel $\sigma$-algebra of subsets of $\mathbb{R}^d$.
For any $V \in \mathcal{B}(\mathbb{R}^d)$, we denote by $\mathscr{B}(V)$ the family of all numerical Borel measurable functions 
\[
f : V \to \overline{\mathbb{R}} := \mathbb{R} \cup \{+\infty\} \cup \{-\infty\}.
\]
For $r > 0$ and $x \in \mathbb{R}^d$, let
\[
B_r(x) := \{y \in \mathbb{R}^d : |x-y| < r\},
\]
where $|\cdot|$ denotes the Euclidean norm, and write $B_r := B_r(0)$.  
For open sets $V \subset \mathbb{R}^d$, let $C_c^\infty(V)$ denote the space of 
smooth functions with compact support in $V$, 
and set $\mathscr{D}(V) := C_c^\infty(V)$.  
Let $C_0(\mathbb{R}^d)$ be the space of continuous functions on $\mathbb{R}^d$ 
vanishing at infinity, equipped with the supremum norm
\[
\|u\|_\infty := \sup_{x \in \mathbb{R}^d} |u(x)|, \qquad u \in C_0(\mathbb{R}^d).
\]
For a given class of numerical functions $\mathscr{A}$, we denote by $p\mathscr{A}$
 the subclass of $\mathscr{A}$ consisting of positive functions, i.e. $[0,+\infty]$-valued functions.  
 Throughout the paper, $j_\varepsilon$ denotes the Friedrichs mollifier.

We let    $d_s:=d/(d-2s)$. 
This is the critical  
 Hardy--Littlewood--Sobolev exponent; that is,
for Riesz potentials
$$
I_{2s} f(x):=\int_{\mathbb{R}^d} \frac{f(y)}{|x-y|^{d-2s}} d y, 
$$
we have 
$$
\|I_{2s} f\|_{M^{d_s}} \leq C\|f\|_{L^{1}(\mathbb{R}^d)},
$$
where $M^{d_s}$ is the Marcinkiewicz space.
Recall that for $p>1$ the Marcinkiewicz space
$M^p$, also denoted by $L^{p,\infty}(\mathbb{R}^d)$, is the
space of all measurable functions $f$ on $\mathbb{R}^d$, identified up to
equality almost everywhere, such that
\[
 \|f\|_{p,\infty}
 :=
 \sup_{\lambda>0}
 \lambda\,
\left[\ell^d\left(\{x\in\mathbb{R}^d:|f(x)|>\lambda\}\right)\right]^{1/p}
 <\infty.
\]
The functional $\|\cdot\|_{p,\infty}$ is equivalent to the norm
\[
 \|f\|_{M^p}
 :=
 \sup_{\substack{E\in\mathcal B(\mathbb R^d)\\
                  0<\ell^d(E)<\infty}}
\frac{\int_E |f(x)|\,dx}{ [\ell^d(E)]^{1/p'}},
\]
with  $p'=p/(p-1)$ (see e.g. \cite[Section~1.1]{Grafakos2014}).

\subsection{Strictly Stable L\'evy Generators}

As mentioned in the introduction, throughout the paper $L$ denotes a strictly 
$2s$-stable L\'evy  operator \eqref{L}--\eqref{scaling_def} satisfying the ellipticity condition \eqref{ellipticity}.  
In the case $s\in (0,1/2)$ we assume that the Green function $G$ of $L$ is strictly positive.
This condition is satisfied if
\begin{equation}
\operatorname{Cone}(\operatorname{Conv}(\operatorname{supp} \lambda))=\mathbb R^d
\end{equation}
where $\operatorname{Cone}(A)$, $\operatorname{Conv}(A)$,  
denote the cone and  convex hull generated by $A$, respectively.
Note that for $s\in [1/2,1)$ the assumptions imposed above guarantee that $G$ is strictly positive 
(see \cite{AshbaughRajputRamaMurthySundberg1992}, \cite{PortVitale1988}). 
Furthermore, by \cite[Proposition I.10]{Bertoin},  $G(x,\cdot)$, $G(\cdot,y)$ are lower semicontinuous 
for fixed $x,y\in\mathbb R^d$.

It is well known that every L\'evy  operator generates a Markov $C_0$-semigroup on $C_0(\mathbb R^d)$,
with the operator core $C_c^2(\mathbb R^d)$ (see \cite[Corollary 2.10]{BSW}).
We denote by $L^*$ the L\'evy operator with representation \eqref{L} obtained by replacing 
$(b,\nu)$ with $(-b,\nu^*)$, where  $\nu^*(dx) := \nu(-dx)$.  
We let $(L,\mathscr{D}(L))$ and $(L^*,\mathscr{D}(L^*))$ be the closures of 
$(L,C_c^2(\mathbb{R}^d))$ and $(L^*,C_c^2(\mathbb{R}^d))$ in $C_0(\mathbb R^d)$, 
respectively, and write $(T_t)$, $(T^*_t)$ for their associated semigroups on $C_0(\mathbb{R}^d)$.  

\begin{remark}
\label{rem.adj}
By \cite[Proposition II.1]{Bertoin}, for any positive $\eta_1, \eta_2 \in C_0(\mathbb R^d)$,
\begin{equation}
\label{eq.dual}
\int_{\mathbb{R}^d} T_t \eta_1 \, \eta_2 = \int_{\mathbb{R}^d} \eta_1 T^*_t \eta_2, \qquad t \geq 0.
\end{equation}
\end{remark}

By \eqref{eq.symbolcomp2} and  \cite[Theorem 3.7]{BF}, the form $C^2_c(\mathbb R^d)\times C^2_c(\mathbb R^d)\ni(u,v)\mapsto \EE^{0}(u,v):=(-Lu,v)$
is closable in $ L^2(\mathbb R^d)$ and its closure, denoted by   $(\EE,\mathscr D(\EE))$, is a regular 
(non-symmetric) Dirichlet form with the domain $\mathscr D(\EE)=H^s(\mathbb R^d)$ (see e.g. \cite{MR}).
In Fourier notation, for $u,v\in \mathscr D(\EE)$,
\begin{equation}
\label{eq.fnot}
\mathcal{E}(u, v)=\int_{\mathbb{R}^d}\psi(\xi) \widehat{u}(\xi) \overline{\widehat{v}(\xi)}\,d \xi.
\end{equation}

By \cite[Theorem I.2.8]{MR} $(T_t)$, $(T^*_t)$ extend from $C_0(\mathbb R^d)\cap L^2(\mathbb R^d)$ to 
 strongly continuous  contraction semigroups $(T_t^{(2)})$, $(T_t^{(2),*})$  on $L^2(\mathbb{R}^d)$.
We denote by $(L^{(2)}, \mathscr{D}(L^{(2)}))$, $(L^{(2),*}, \mathscr{D}(L^{(2),*}))$  the operators generated by 
$(T_t^{(2)})$ and $(T_t^{(2),*})$, respectively.

\subsection{L\'evy processes and related semigroups}
\label{sec.fprs}

It is well known that for any L\'evy generator $(L,\mathscr D(L))$ there exists a L\'evy process $(Y_t)_{t\ge 0}$
on a complete probability space $(\Omega_0,\mathcal F,\mathbb P)$ that is associated with  $L$, i.e.
\[
\mathbb Ef(x+Y_t)=T_tf(x),\quad t\ge 0,x\in\mathbb R^d,\, f\in C_0(\mathbb R^d).
\]
Here and in what follows we use the standard  notation of probability theory: 
 for any measurable function $Z$ on $\Omega_0$,
\[
\mathbb EZ:=\int_{\Omega}Z(\omega)\,\mathbb P(d\omega).
\]
Recall that a c\`adl\`ag process $(Y_t)$ is called a L\'evy process if
\begin{enumerate}
\item[(i)] $Y_0=0$, and   $Y_t-Y_s$,  $Y_{t-s}$ admit the same distribution  for any $0\le s\le t$,
\item[(ii)] $Y_{t_k}-Y_{t_{k-1}},\dots,Y_{t_2}-Y_{t_1}$ are independent for any $0\le t_1\le \dots\le t_k$ and $k\ge 1$.
\end{enumerate}
It is convenient  to introduce the canonical realization of the L\'evy process $(Y_t)$.
Let $\Omega$ denote the set of all functions $\omega:[0,\infty)\to \mathbb R^d$
that are  right-continuous on $[0,\infty)$ and possess left limits on $(0,\infty)$ (c\`adl\`ags).
We equip $\Omega$ with the Skorokhod metric $\rho$ (see  \cite[Section 12]{bil}). 
Then $(\Omega,\rho)$ is a Polish space.
Next, let 
\[
X_t(\omega):=\omega(t),\quad \omega\in \Omega,\, t\ge 0,\quad \mathcal F^X_t:=\sigma(X_s: s\le t).
\]
We define the family  $(\mathbb P_x,\, x\in \mathbb R^d)$ 
of Borel probability measures on $\Omega$ by
\[
\mathbb P_x(X_\cdot \in B):= \mathbb P(x+Y_\cdot\in B),\quad B\in\mathcal B(\Omega).
\]
Observe that 
\begin{itemize}
\item for any $f\in C_0(\mathbb R^d)$
\[
T_tf(x)=\mathbb E_xf(X_t),\quad x\in\mathbb R^d,\, t\ge 0,
\]
\item for all $x\in\mathbb R^d$, $s,t\ge 0$, we have  $\mathbb E_{X_t}f(X_s)=\mathbb E_x[f(X_{s+t})|\FF_t^X]$.
\end{itemize}
The second condition says that $(\mathbb P_x,\, x\in \mathbb R^d)$ is a {\em Markov family}.

Let $(\mathbb P^{*}_x,\, x\in \mathbb R^d)$ 
denote the canonical Markov family  related to $L^*$.
In what follows for   any  $f\in p\mathscr B(\mathbb R^d)$  we let
\[
P_tf(x):=\mathbb E_x f(X_t),\quad P^{*}_tf(x):=\mathbb E^{*}_x f(X_t),\qquad x\in \mathbb R^d.
\]
For any  $f\in\mathscr B(\mathbb R^d)$, we let
$P_tf(x):= P_tf^+(x)-P_tf^-(x)$ if $P_tf^+(x)<\infty$ or $P_tf^-(x)<\infty$, and zero otherwise.
We adopt   an analogous convention  for $P^{*}_tf$. 
We also let for any    $f\in p\mathscr B(\mathbb R^d)$
\[
R_\alpha f(x):= \int_0^\infty e^{-\alpha t} P_t f(x)\,dt, \quad R^{*}_\alpha f(x):= \int_0^\infty e^{-\alpha t} P^{*}_t f(x)\,dt,\qquad x\in \mathbb R^d.
\]
For each   $R_\alpha f$ and $R^{*}_\alpha f$, with $f\in\mathscr B(\mathbb R^d)$, we adopt the convention 
analogous to that for  $(P_tf), (P^{*}_tf)$.
Finally, we denote  $R:=R_0, R^{*}:=R^{*}_0$. 

Let $V\in\mathcal B(\mathbb R^d)$ and $(S_t)$ be  a semigroup of  positive linear operators on $p\mathscr B(V)$.
Recall (see \cite[Section II]{BG}) that   a positive function $u\in \mathscr B(V)$
is called $\alpha$-excessive with respect to  $(S_t)$ provided that:
\begin{enumerate}
\item[a)] $e^{-\alpha t}S_tu(x)\le u(x),\, t\ge 0,\, x\in V$,
\item[b)] $\lim_{t\to 0^+}S_tu(x)=u(x),\, x\in V$.
\end{enumerate}
If $\alpha=0$ we simply say that $u$ is $(S_t)$-excessive.
By \cite[Theorem VI.1.4]{BG} for any $\alpha\ge 0$  there exists a  unique
function $G_\alpha \in p\mathscr B(\mathbb R^d\times \mathbb R^d)$
such that for any $f\in p\mathscr B(\mathbb R^d)$,
\begin{equation}
\label{eq.GGG}
R_\alpha f(x)=\int_{\mathbb R^d} G_\alpha(x,y)f(y)\, dy,\quad R^{*}_\alpha f(x)=\int_{\mathbb R^d} G_\alpha(y,x)f(y)\, dy,\quad x,y\in\mathbb R^d,
\end{equation}
and  $G_\alpha(\cdot,y)$, $G_\alpha(x,\cdot)$ are  $\alpha$-excessive with respect to $(P_t)$, $(P^{*}_t)$, respectively
for any $x,y\in\mathbb R^d$.
The function $G_\alpha(\cdot,\cdot)$ is called $\alpha$-Green's function for $-L$. In case $\alpha=0$ we simply write 
 $G(\cdot,\cdot)$.

\subsection{Killed L\'evy processes and related Green functions}

For any Borel  set $B\subset\mathbb R^d$, we let 
\begin{equation}
\label{eq.hitt}
\tau_B(\omega):=\inf\{t>0: X_t(\omega)\notin B\}.
\end{equation}
It is known  (see e.g.  \cite[Theorem III.3.3]{BG} applied to $\mathbb P_x$ and  $M_t:=\mathbf1_{[0,\tau_V)}(t)$)
that   for open $V\subset\mathbb R^d$,
\begin{equation}
\label{eq.pdef}
\begin{split}
&P^{V}_tf(x):=\mathbb E_x [f(X_t)\mathbf 1_{\{t<\tau_V\}}],\quad P^{*,V}_tf(x):=\mathbb E^{*}_x [f(X_t)\mathbf 1_{\{t<\tau_V\}}]
\end{split}
\end{equation}
define Markov semigroups.
We let
\begin{equation}
\label{eq.rdef}
R^{V}_\alpha f(x):=\mathbb E_x \int_0^{\tau_V} e^{-\alpha t}f(X_t)\, dt,\quad
R^{*,V}_\alpha f(x):=\mathbb E^{*}_x \int_0^{\tau_V} e^{-\alpha t}f(X_t)\, dt,\quad x \in V.
\end{equation}
Observe that for  $f\in p\mathscr B(\mathbb R^d)$, $P^{V}_tf\le P_tf$ for all $t>0$.
By \cite[Theorem 3.5.7]{Oshima}, $(P^{V}_t), (P^{*,V}_t)$ 
can be extended from $C_b(V)\cap L^2(V)$ to  strongly continuous  contraction semigroups on $L^2(V)$, 
that we shall denote by  $(T^{(2),V}_t), (T^{(2),*,V}_t)$, respectively.
We denote by $L^{(2)}_{V}, L^{(2),*}_{V}$ respective generators of  $(T^{(2),V}_t), (T^{(2),*,V}_t)$, and by 
$(G^{(2),V}_\alpha), (G^{(2),*,V}_\alpha)$ respective resolvents.

By \cite[Theorem 3.4]{SW1} $(R^{V}_\alpha)_{\alpha\ge 0}$  is strongly Feller. 
Furthermore, by \cite[Theorem 3.5.7]{Oshima} 
\[
\int_{\mathbb R^d} R^{V}f\,g=\int_{\mathbb R^d} f\,R^{*,V}g,\quad f,g\in\mathscr B_b(\mathbb R^d).
\]
By \cite[Theorem VI.1.4]{BG},  for any $\alpha\ge 0$ there exists $\alpha$-Green's function 
\begin{equation}
\label{eq.gvvve}
G_{V,\alpha}:V\times V\to [0,\infty],
\end{equation}
related to  $(R^{V}_\alpha)_{\alpha\ge 0}$, and
according to \cite[Definition VI.1.2]{BG}   $(R^{V}_\alpha)$ and $(R^{*,V}_\alpha)$ 
are {\em in duality relative to the Lebesgue measure}.
Here $G_{V,\alpha}(x,y)$ plays the role of $u_\alpha(x,y)$ considered in \cite[Definition VI.1.2]{BG}.
 We set 
\[
G_{V}(x,y):=G_{V,0}(x,y).
\]
In what follows we regard functions $G_{V,\alpha}$ as functions on the whole $\mathbb R^d\times\mathbb R^d$
 by setting them equal to zero  outside $V\times V$.

\section{Green operator and its properties}
\label{sec3}
Let $V$ be an open  subset of $\mathbb R^d$. For any Borel positive measure $\mu$ on $V$ and $\alpha\ge 0$ we let
\begin{equation}
\label{dual.res2}
R^{V}_\alpha\mu(x):=\int_{V}G_{V,\alpha}(x,y)\,\mu(dy),\quad  R^{*,V}_\alpha \mu(x):=\int_{V}G_{V,\alpha}(y,x)\,\mu(dy),\quad x\in V.
\end{equation}
We also use the shorthand notation $R^{V}\mu=R^{V}_0\mu$, $R^{*,V} =R^{*,V}_0$.
For any positive Borel functions $f,g$ on $\mathbb R^d$ and positive Borel measure $\mu$ on $\mathbb R^d$ we let
\[
(f,g):=\int_{\mathbb R^d}fg,\quad (\mu,f):=\int_{\mathbb R^d}f\,d\mu.
\]

By the very definition, $G_V(\cdot,y)$ is $(P^V_t)$-excessive for any $y\in V$ and 
$G_V(x,\cdot)$ is $(P^{*,V}_t)$-excessive for any $x\in V$. Therefore $R^V\mu$ (resp. $R^{*,V}\mu$) is 
$(P^V_t)$-excessive  (resp. $(P^{*,V}_t)$-excessive) for any positive Borel measure $\mu$ on $V$.

\begin{lemma}\label{trunc-sobolev}
Let $\mu$ be a positive bounded Borel measure on $V$, and let $u := R^V\mu$. Then for every $k\ge 0$, 
$u\wedge k\in H_0^s(V)$ and satisfies
\begin{equation}
\label{trunc-measure-bound}
\mathcal{E}(u\wedge k, u\wedge k) \le k \mu(V).
\end{equation}
\end{lemma}

\begin{proof}
For $\alpha > 0$, define $v_\alpha := \alpha R_\alpha^V \mu$. 
Let $u_\alpha := R^Vv_\alpha$. By  \cite[Theorem 4.2]{KR:CM}, $u_\alpha\wedge k  \in H_0^s(V)$ and
\begin{equation}\label{alpha-energy-bound}
    \mathcal{E}(u_\alpha\wedge k, u_\alpha\wedge k) \le k\int_Vv_\alpha.
\end{equation}
Since the dual resolvent is sub-Markovian, we have
\begin{equation}\label{mu-alpha-bound}
\int_V v_\alpha=(v_\alpha,1)= (\alpha R_\alpha^V \mu,1)= (\mu,\alpha R^{*,V}_\alpha 1) \le\mu(V).
\end{equation}
Combining \eqref{alpha-energy-bound} and \eqref{mu-alpha-bound} yields a uniform bound on the energy
\begin{equation}\label{uniform-energy}
    \mathcal{E}(u_\alpha\wedge k, u_\alpha\wedge k) \le k \mu(V).
\end{equation}
On the other hand, using the resolvent identity, we observe that
\begin{equation*}
    u_\alpha = R^V (\alpha R_\alpha^V \mu) = \alpha R_\alpha^V (R^V \mu) = \alpha R_\alpha^V u.
\end{equation*}
Since $u = R^V \mu$ is a $(P^V_t)$-excessive function, the sequence $\alpha R_\alpha^V u$ increases pointwise to $u$ as $\alpha \to \infty$. Consequently, $u_\alpha\wedge k \nearrow u\wedge k$.
Because of the uniform energy bound \eqref{uniform-energy}, the sequence $(u_\alpha\wedge k)_{\alpha > 0}$
 is bounded in the Hilbert space $H_0^s(V)$. By extracting a weakly convergent subsequence and using 
Fatou's lemma for forms, we conclude that the pointwise limit $u\wedge k$ belongs to $H_0^s(V)$ and satisfies
\begin{equation*}
\mathcal{E}(u\wedge k, u\wedge k)\le\liminf_{\alpha \to \infty} \mathcal{E}(u_\alpha\wedge k, u_\alpha\wedge k)\le k\mu(V).
\end{equation*}
This completes the proof.
\end{proof}

\begin{lemma}
\label{lm.duoper}
Let   $\mu$ be a positive bounded Borel measure  on $V$.
Then  $R^{V}\mu\in M^{d_s}$ and  for any $\eta\in C_c^2(V)$,
\[
-(R^{V}\mu,L^{*}\eta)=(\mu,\eta).
\]
Furthermore, there exists a constant $C>0$, independent of $\mu$ and $V$, such that
\begin{equation}
\label{eq.margr}
\|R^V\mu\|_{M^{d_s}}\le C\mu(V).
\end{equation}
\end{lemma}
\begin{proof}
In view of Proposition~\ref{prop.green}, it suffices to establish \eqref{eq.margr} when \(V\) is bounded. 
Set \(u:=R^V\mu\). 
Lemma~\ref{trunc-sobolev} yields \(u\wedge k\in H_0^s(V)\), together with the energy estimate \eqref{trunc-measure-bound}. 
Combining \eqref{eq.symbolcomp1} and  \eqref{eq.fnot}, we obtain
\[
c\int_{\mathbb{R}^d} \int_{\mathbb{R}^d} \frac{|(u\wedge k)(x)-(u\wedge k)(y)|^2}{|x-y|^{d+2s}} \, dx\,dy
=C\int_{\mathbb{R}^d}|\xi|^{2 s}|\widehat{(u\wedge k)}(\xi)|^2 \,d \xi\leq \mathcal{E}(u\wedge k, u\wedge k),
\]
which together with Sobolev embedding (see e.g. \cite[Theorem 7.6]{Leoni2023}) and \eqref{trunc-measure-bound} implies that 
\[
\|u\wedge k \|^2_{L^{2{d_s}}(V)}\leq C k\mu(V).
\]
Thus 
\[
\ell^d(\{u\ge k\})\leq Ck^{-{d_s}}[\mu(V)]^{d_s},\quad k>0,
\]
showing \eqref{eq.margr}.
It remains to verify the asserted distributional identity. 
For \(\eta\in C_c^2(V)\), we have \(L^*\eta\in C_b(\mathbb R^d)\) and \(L^*\eta=L_V^{(2),*}\eta\). 
Hence 
\[
R^{*,V}L^*\eta = R^{*,V}L_V^{(2),*}\eta = G^{(2),*,V}L_V^{(2),*}\eta = \eta \quad\text{a.e. in }V.
\]
Both the first and the last members of this chain admit continuous versions; the identity therefore holds pointwise throughout 
\(V\). Duality now gives \[ (R^V\mu,L^*\eta) = (\mu,R^{*,V}L^*\eta) = (\mu,\eta), \] and the proof is complete.
\end{proof}

\begin{corollary}
\label{green-weak}
We have $\sup_{x\in \mathbb R^d}\|G(x,\cdot)\|_{M^{{d_s}}}<\infty$ and $\sup_{y\in \mathbb R^d}\|G(\cdot,y)\|_{M^{{d_s}}}<\infty$. 
\end{corollary} 
\begin{proof}
It follows directly from \eqref{eq.margr} since $G(x,y)=(R\delta_y)(x)$, while 
 $G(x,y)=(R^{*}\delta_x)(y)$.
\end{proof}

By Proposition~\ref{prop.green}
\[
G_V(x,y)\leq G(x,y),\quad x,y\in \mathbb{R}^d.
\]
Since the values of $G$ depend only on the difference of the arguments, we shall use the notation $g(x-y):=G(x,y)$. Define
\[
g_V(z):=\mathbf{1}_{V-V}(z)g(z),\quad z\in \mathbb{R}^d,\ V-V:=\{x-y\ |\ x,y\in V\}.
\]
Thus the following inequality holds
\[
G_V(x,y)\leq g_V(x-y),\quad x,y\in \mathbb{R}^d,
\]
which implies that, for $f\in p\mathscr B(\mathbb R^d)$, 
\begin{equation}\label{conv-ineq}
R^Vf\leq g_V\ast (\mathbf1_V f).
\end{equation}

\begin{corollary}\label{green-young}
Let $D$ be an open bounded subset of $\mathbb R^d$ and $h\in L^q(D)$ for some $q\geq 1$. Then the following estimates hold on $L^r(D)$ norms $(1\leq r\leq\infty)$:
\begin{enumerate}
\item $\|\,R^D[h]\,\|_\infty\lesssim \|h\|_q$ for $q>d_s'$;
\item $\|\,R^D[h]\,\|_r\lesssim \|h\|_q$ for $\frac{1}{q}\leq\frac{1}{r}+\frac{1}{d_s'}$, when $1<q<d_s'$;
\end{enumerate}
\end{corollary}
\begin{proof}
Since $D$ is bounded, the results follow from  \eqref{conv-ineq},  Corollary~\ref{green-weak}
and Young’s convolution inequality.
\end{proof}

\begin{lemma}
\label{lm2.2}
Let $V$ be a  bounded open subset of $\mathbb R^d$, let   $(\mu_n)$ be a sequence of   bounded 
positive Borel measures on $V$, and let $\mu$ be such a measure. Suppose that $\mu_n\to \mu $ in $[C_b(V)]^*$. Then
\begin{enumerate}
\item[(i)] $R^V\mu_n+R^V\mu<\infty$ q.e., $n\ge 1$,
\item[(ii)] $R^V\mu_n\to R^V\mu$ a.e. along a subsequence.
\end{enumerate}
\end{lemma}
\begin{proof}
(i) By Corollary~\ref{green-young} for any bounded positive measure $\gamma$ on $\mathbb R^d$, we have 
\[
(R^V\gamma,1)=(\gamma,R^{*,V}1)\le \|R^{*,V}1\|_\infty\gamma(V)<\infty.
\]
The result follows from \cite[Proposition VI.2.3]{BG}.

(ii) By Deny's theorem  there exists a subsequence (not relabeled) 
and a $(P^V_t)$-excessive function $w$  such that $R^V\mu_n\to w$ a.e. 
Let $\xi\in \mathscr B_b(V)$. By the assumptions made  $R^{*,V}\xi\in C_b(V)$ (see also \cite[Theorem 3.4]{SW1}).
Thus 
\[
(R^{V}\mu_n,\xi)=(\mu_n,R^{*,V}\xi)\to (\mu,R^{*,V}\xi)=(R^{V}\mu,\xi).
\]
In other words, $R^{V}\mu_n$ converges weakly to $R^V\mu$ in $L^1(V)$.
By the Dunford-Pettis theorem the sequence $(R^{V}\mu_n)_{n\ge 1}$ is uniformly integrable.
As a result, by Vitali's theorem,  $w=R^V\mu$.
\end{proof}

\section{Deny's theorem and compactness of Green operators}
\label{sec4}
\begin{definition}
We say that a positive  function $u$ is $(P^D_t)$-{\em supermedian}  if the following condition is satisfied
\[
u(x) \geq \left[\alpha R^D_{\alpha}u\right](x),\quad x\in D.
\]
\end{definition}
\begin{theorem}[Deny's theorem, see \cite{deny}]
\label{th-denys}
Let $(u_n)$ be a sequence of  $(P^D_t)$-supermedian functions.
Then there exists a subsequence $(u_{n_k})$ and a $(P^D_t)$-excessive function $u$ such that 
$(u_{n_k})$ converges to $u$ almost everywhere in $D$.
\end{theorem}
\begin{proof}
We first treat the case in which \((u_n)\) is uniformly bounded. 
The general assertion then follows by applying the bounded case to the truncations \((u_n\wedge k)\), \(k\in\mathbb N\), 
and using a diagonal subsequence argument. Since the resolvent is sub-Markovian, each \(u_n\wedge k\) remains supermedian. 
After normalization, we may assume that \(0\leq u_n\leq1\). By weak-star compactness of the unit ball of \(L^\infty(D)\), 
a subsequence, still denoted by \((u_n)\), converges in \(\sigma(L^\infty,L^1)\) to a measurable function \(\widetilde u\). 
Since \(G_{D,\alpha}(x,\cdot)\in L^1(D)\) (see Lemma~\ref{lm.duoper}), this convergence implies 
\[
R_\alpha^D u_n(x)\longrightarrow R_\alpha^D\widetilde u(x)
\]
for every \(x\in D\). Passing to the limit in the supermedian inequality shows that \(\widetilde u\geq\alpha R_\alpha^D\widetilde u\) 
almost everywhere. Proposition~2.4 of \cite{BCR} therefore implies that 
\[
u(x):=\lim_{\alpha\to\infty}\alpha R_\alpha^D\widetilde u(x)
\]
is excessive and agrees with \(\widetilde u\) almost everywhere.
It remains to show pointwise convergence of $(u_n)$.
Since $u_n$ converges weakly to $u$ and $\alpha R^D_{\alpha} u_n \le u_n$, we have
$$
\alpha R^D_{\alpha} u \le \liminf_{n\to\infty} u_n \quad \text{a.e.}
$$
Letting $\alpha \to \infty$, we establish that $u \le \liminf_{n\to\infty} u_n$ almost everywhere. 
Consider the sequence $w_n = u_n - u$. We know $w_n \to 0$ weakly in $\sigma(L^1, L^\infty)$ and $\liminf w_n \ge 0$.
We shall show that $w_n\to 0$ in $L^1(D)$. We write $|w_n|=w_n+2w_n^-$. The equality
\[
\int_D w_n(x)\ dx \to 0,\quad n\to\infty
\]
follows from weak convergence. On the other hand $\liminf w_n \ge 0$ implies that $w_n^-\to 0$ pointwise. 
By the dominated  convergence theorem 
\[
\int_D w^-_n(x)\ dx \to 0,\quad n\to\infty,
\]
which completes the proof.
\end{proof}

\begin{remark}
\label{rem.deny}
Observe that for any $f\in p\mathscr B(\mathbb R^d)$, the function $g:=R^D[f]$ is $(P^D_t)$-supermedian 
by the resolvent identity. Thus if we have the sequence $(f_n)\subset p\mathscr B(\mathbb R^d)$, 
Theorem~\ref{th-denys} guarantees the existence of a 
subsequence $R^D[f_{n_k}]$ convergent almost everywhere on $D$.
\end{remark}

\begin{lemma}
\label{green-l1}
For any $p\in [1,d_s)$ the operator $T_p: L^p(D)\mapsto L^p(D)$ given by $T_p(f):=R^D[|f|^p]$ is continuous and compact.
\end{lemma}
\begin{proof}
By Corollary~\ref{green-young}, we deduce that $T_p$ is continuous and 
for any bounded  $B\subset L^p(D)$  the family $\{(T_pf)^p: f\in B\}$ is uniformly integrable.
This, together with Remark~\ref{rem.deny},  proves the assertion. 
\end{proof}

\section{Distributional and classical  solutions}
\label{sec5}

\begin{definition}
For $\delta>0$ and open set $D$ we put
\[
o_\delta(D):=\{x\in D\ |\ \text{dist}(x,D^c)>\delta\}.
\]
\end{definition}

\subsection{Distributional solutions}

When studying   equations like   \eqref{eq1.0} in the sense of distributions 
one has to  define  $Lu$  properly as a distribution on $D\setminus\{0\}$.  
In contrast  to  local operators, it is not enough  to impose  local integrability of $u$ on $D\setminus\{0\}$
but rather one needs control of the integrability properties of $u$ on the whole $\mathbb R^d$. 
In order to have $Lu$ well defined as a distribution, we introduce below the space $\mathscr T_\nu(V)$
that depends on the L\'evy measure $\nu$ in \eqref{L}.

\begin{definition}
\label{def.space1}
Fix an open set $V\subset \mathbb R^d$.
For any compact $K\subset V$, we let   $r_{K,V}:=(2\text{dist}(K,V^c))\wedge 1$ and 
\[
\rho_{K,V}(x):= \nu(B^c_{r_{K,V}}\cap (K-x))+\nu(B^c_{r_{K,V}}\cap (x-K)).
\]
 We further define the linear space 
\[
\mathscr T_\nu(V):= \bigcap _{K\subset\subset V} L^1_{\rho_{K,V}}(\mathbb R^d).
\]
\end{definition}

\begin{remark}
Observe that in case $\lambda(d \theta) \asymp \sigma(d \theta)$ (see \eqref{eq.specme}) 
or, equivalently, 
$$
\nu(dz) \asymp|z|^{-d-2 s}\, d z,
$$
we have
$$
\mathscr{T}_\nu(V) =  L^1(\mathbb{R}^d; (1+|x|)^{-d-2 s} d x).
$$
On the other hand, in the case of the symmetric stable cylindrical process, i.e. if $\psi(\xi)=|\xi_1|^{2s}+ \dots +|\xi_d|^{2s}$ (see \eqref{eq.symb1}-\eqref{eq.symb2}), we have 
\[
\rho_{K,V}(x) \asymp \sum_{j=1}^d \frac{\ell^1\left(K_{x, j}\right)}{\left(1+\left|x_j\right|\right)^{1+2 s}},
\]
where $K_{x, j}:=\{t \in \mathbb{R}: x+t e_j \in K\}$.
Consequently, there is no inclusion between the corresponding spaces $\mathscr T_\nu(V)$.
\end{remark}

\begin{remark}
\label{rem.incl123}
We record two elementary consequences of the definition of
$\mathscr T_\nu(V)$.
First, for every open set $V\subset \mathbb R^d$ one has the following inclusions
\[
L^1(\mathbb R^d)\subset  \mathscr T_\nu(V)\subset L^1_{\mathrm{loc}}(V),
\]
if one considers $L^1_{\mathrm{loc}}(V)$ as a space of measurable functions  on  $\mathbb R^d$
that are locally integrable on $V$. Second, if $V\subset\mathbb R^d$ is open and bounded, and  $x_0\in V$, then
\[
\mathscr T_\nu(V\setminus\{x_0\})\subset L^1_{\mathrm{loc}}(V).
\]
Indeed, the first inclusion is trivial since each $\rho_{K,V}$ is bounded.
Next, for any $C\subset\subset  V$ one can find compact $K\subset V$ and 
$R>r_{K,V}$ such that $\{z: r_{K,V}<|z|<R\}\subset B^c_{r_{K,V}}\cap (K-x)$ for $x\in C$.
Consequently, for $x\in C$,
\[
\begin{split}
\rho_{K,V}(x)&\ge \nu(B^c_{r_{K,V}}\cap (K-x))\ge \nu\bigl(\{z\in\mathbb R^d:r_{K,V}<|z|<R\}\bigr)=\lambda(\mathbb S^{d-1})\int_{r_{K,V}}^{R}\frac{dr}{r^{1+2s}}, 
\end{split}
\]
which establishes the second inclusion of the first assertion. 
As to the second one, it remains only to control the possible singularity at $x_0$. Choose
$a>0$ so small that $\overline B(x_0,5a)\subset V$ and $2a<1$, and put
\[
        K:=\{y\in\mathbb R^d: a\le |y-x_0|\le 4a\}.
\]
Then $K\subset\subset V\setminus\{x_0\}$ and $r_{K,V}=2a$. 
If $x\in B(x_0,a/4)$, then for every $\theta\in\mathbb S^{d-1}$ and every
$r\in(9a/4,3a)$ we have
\[
x+r\theta\in K.
\]
Hence
\[
\begin{split}
\rho_{K,V}(x)&\ge \nu(B^c_{2a}\cap (K-x))\ge\lambda(\mathbb S^{d-1})\int_{9a/4}^{3a}\frac{dr}{r^{1+2s}},\quad x\in B(x_0,a/4).
\end{split}
\]
Thus every $u\in\mathscr T_\nu(V\setminus\{x_0\})$ is integrable in a
neighbourhood of $x_0$. 
\end{remark}

\begin{lemma}
\label{lm.test1}
Let   $u\in\mathscr T_\nu(V)$.  Then 
\[
\int_{\mathbb R^d}|u|\,|L\eta|<\infty,\quad \eta\in C^2_c(V).
\]
\end{lemma}
\begin{proof}
Fix  $\eta\in C_c^2(V)$ and let    $r:=\text{dist}(supp[\eta],V^c)$. 
Observe that for any $x\in V$,
\[
\begin{split}
\left|\int_{B_{\frac r2\wedge 1}}(\eta(x)-\eta(x+y)-y\nabla\eta(x))\,\nu(dy)\right|&
\le \int_{B_{\frac r2\wedge1}}|y|^2\sup_{\theta\in [0,1]}|D^2\eta(x+\theta y)|\,\nu(dy)
\\&
\le \sup_{z\in B(x,1)}|D^2\eta(z)| \int_{B_{1}}|y|^2\,\nu(dy).
\end{split}
\]
On the other hand  
\[
\begin{split}
&\left|\int_{B^c_{\frac r2\wedge 1}}(\eta(x)-\eta(x+y)-y\nabla\eta(x))\,\nu(dy)\right|
\le c_\nu |\nabla\eta(x)| + \left|\int_{B^c_{\frac r2\wedge1}}(\eta(x)-\eta(x+y))\,\nu(dy)\right|\\&
\qquad\qquad\qquad\qquad\le c_\nu |\nabla\eta(x)|+  \nu(B^c_{\frac r2\wedge1})\eta(x)+\int_{B^c_{\frac r2\wedge1}}|\eta(x+y)|\,\nu(dy).
\end{split}
\]
Consequently, for $x\in\mathbb R^d$,
\[
|L\eta(x)|\le c_\nu\left( \sup_{z\in B(x,1)}|D^2\eta(z)|+|\nabla\eta(x)|+\eta(x)\right)+\|\eta\|_\infty \nu(B^c_{\frac r2\wedge1}\cap (\text{supp}[\eta]-x)).
\]
The first term  on the right-hand side of the above inequality is  compactly supported in $V$, and 
so multiplied by $u$ is integrable. 
The second term multiplied by $u$ is integrable by the very definition of the space  $\mathscr T_\nu(V)$.
\end{proof}

Let $V$ be an open  subset of $\mathbb R^d$ and $\mu$ be a Radon measure on $V$. 
Consider the equation
\begin{equation}
\label{eq1.00}
-Lu=\mu\quad \text{in } V.
\end{equation}

\begin{definition}\label{distributional-solution}
We say that $u\in \mathscr T_\nu (V)$ is a  {\em distributional solution} 
to \eqref{eq1.00} provided that 
\begin{equation}
\label{eq1.3}
-\int_{\mathbb R^d}u \,L^*\eta=\int_{\mathbb R^d}\eta\,d\mu,\quad\eta\in \mathscr D(V).
\end{equation}
We then write $-Lu=\mu$ in $\mathscr D'(V)$.
\end{definition}

\begin{remark}
Observe that the integral on the right-hand side of \eqref{eq1.3} is in fact over $V$, 
while the one on the left-hand side, over the whole space $\mathbb R^d$. In the next remark 
we show that the integral $\int_{V^c}u \,L^*\eta$ as a functional of $\eta\in C^2_c(V)$ is positive,
and as a result is a positive measure that we denote by $\nu^{*,V}_u$. Thus, we equivalently have 
\begin{equation}
\label{eq1.3e}
-\int_{V}u \,L^*\eta=\int_{V}\eta\,d(\mu+\nu^{*,V}_u),\quad\eta\in \mathscr D(V).
\end{equation}
\end{remark}

\begin{remark}
\label{rem.imp}
Define the measure $\nu^*_x(B):=\nu(x-B)$, $x\in\mathbb R^d$ and $B\in\mathcal B(\mathbb R^d)$.
We further define the positive Radon measure $\nu^{*,V}_u(dy)$ on $V$ by 
\begin{equation}
\label{eq.defnust}
\nu^{*,V}_u(dy):= \int_{V^c}u(x)\nu^*_x(dy)\,dx.
\end{equation}
Now, notice that
\[
\int_{\mathbb R^d}u \,L^*\eta=\int_{V}u \,L^*\eta+\int_V \eta\,d\nu^{*,V}_u.
\]
Consequently, by Lemma~\ref{lm.test1}, $\nu^{*,V}_u$ is a positive Radon measure on $V$.
Moreover, $\nu^{*,V}_u(K)=0$ whenever $K$ is a closed polar subset of $V$
(see  \cite[Lemma 2.16]{BSW}; see also Proposition~\ref{prop.green}(3) for the definition of a polar set).
\end{remark}

The following two results will be needed later. We denote $\nu_x(dy):=\nu^*_x(-dy)$.

\begin{lemma}
\label{lm.ch1}
Let $u, v\in p\mathscr B(E)$ and
$supp[u]\cap supp[v] = \emptyset$. Then
\begin{equation}
\label{eq2.weq1}
\int_{\mathbb R^d}\int_{\mathbb R^d} u(x)v(y)\nu_x(dy)\,dx=
\int_{\mathbb R^d}\int_{\mathbb R^d} u(y)v(x)\nu^*_x(dy)\,dx.
\end{equation}
\end{lemma}
\begin{proof}
The asserted  equality for $u, v\in C_c^2(\mathbb R^d)$ is nothing more than  
\[
(Lu,v)=(u,L^*v). 
\]
Now, the general result is a consequence of the monotone class theorem.
\end{proof}

\begin{lemma}
\label{lm.hr1}
Let $V\subset\subset D$ be  an  open set,   and $u\in p\mathscr B(\mathbb R^d)$. 
Then
\[
R^{V}\nu^{*,D}_u(x)=\mathbb E_x [\mathbf1_{D^c} u(X_{\tau_V})],\quad x\in V.
\]
\end{lemma}
\begin{proof}
The left-hand side of the asserted identity equals
\[
\int_VG_V(x,y)\,\nu^{*,D}_u(dy)= \int_V\int_{D^c} G_V(x,y)u(z)\,\nu^{*}_z(dy)\,dz.
\]
On the other hand, by the Ikeda-Watanabe formula (see \cite{IW}) the right-hand side of the asserted identity equals 
\[
\int_{D^c}u(y)\, P_V(x,dy)=\int_{D^c}u(y)\int_V G_V(x,z)\,\nu_z(dy)\,dz,
\]
where $P_V$ is the Poisson kernel.
Now, the result follows from Lemma~\ref{lm.ch1}.
\end{proof}

\subsection{Classical solutions}

Let $f\in p\mathscr B(\mathbb R^d)$.
Consider the equation
\begin{equation}
\label{eq.class}
-Lu=f\quad\text{in }D.
\end{equation} 

\begin{definition}
We say that  $u\in p\mathscr B(\mathbb R^d)$ is a classical solution to 
\eqref{eq.class}  if $u\in C^1(D)$,
and 
\[
\int_{\mathbb R^d}|u(x+y)-u(x)-y\nabla u(x)\mathbf1_{B(0,1)}(y)|\,\nu(dy)<\infty,\quad -Lu(x)=f(x),\quad x\in D.
\]
\end{definition}

\begin{remark}
\label{rem.cldist}
Observe that for  $f\in L^1_{loc}(D)$,  if $u\in p\mathscr B(\mathbb R^d)\cap \mathscr T_\nu(D)$ is a classical solution to 
\eqref{eq.class}, then $u$ is a distributional solution to \eqref{eq.class}.
Indeed, first note that $-L(u\ast j_\varepsilon)=-j_\varepsilon\ast (Lu)=j_\varepsilon\ast f$ in $o_\varepsilon(D)$.
Thus, for $\eta\in C_c^\infty(o_\varepsilon(D))$, by using Lemma~\ref{lm.test1}, 
\[
(j_\varepsilon\ast f,\eta)=-(L(j_\varepsilon\ast u),\eta)= -(j_\varepsilon\ast u,L^*\eta).
\]
Therefore, $(j_\varepsilon\ast f,\eta)= -(j_\varepsilon\ast u,L^*\eta)$, $\eta\in C_c^\infty(o_\varepsilon(D))$,
and letting $\varepsilon\to 0$ gives the claim.

 In particular. if $0\in D$, $u\in C^1(D\setminus\{0\})\cap \mathscr T_\nu(D\setminus\{0\})$ is a classical solution to 
\[
-Lu=u^p\quad\text{in }D\setminus\{0\},
\] 
then $-Lu=u^p$ in $\mathscr D'(D\setminus\{0\})$.
\end{remark}

\begin{definition}
Let $0\in D$. We say that  $u\in p\mathscr B(\mathbb R^d)$ is a classical solution to 
\begin{equation}
\label{eq1.1class1}
\begin{cases}
-Lu=u^p,\quad &\text{in}\quad D\setminus\{0\}, \\
u=0,\quad &\text{in}\quad\mathbb{R}^d\setminus D,
\end{cases}
\end{equation}
if $u\in C^1(D\setminus\{0\})\cap C(\mathbb R^d\setminus\{0\})$,
and 
\[
\int_{\mathbb R^d}|u(x+y)-u(x)-y\nabla u(x)\mathbf1_{B(0,1)}(y)|\,\nu(dy)<\infty,\quad -Lu(x)=u^p(x),\quad x\in D\setminus\{0\}.
\]
\end{definition}

\section{Isolated   singularities for semilinear equations}
\label{sec6}

We begin with a comment on the relation between the present argument and the B\^ocher-type 
theorem obtained by the second author in \cite{KlimsiakBocher2026} for elliptic equations driven by 
drift-perturbed L\'evy operators. The result proved below is analogous in spirit, but the proof is different and considerably 
more direct. This simplification is made possible by the fact that \(L\) is the generator of a strictly stable L\'evy process. 
In particular, the operator is translation invariant and enjoys an exact scaling property, which allows us to work directly 
with mollifications, killed resolvents, and the Riesz decomposition. By contrast, the operator considered in 
\cite{KlimsiakBocher2026} contains an \(x\)-dependent drift term and therefore lacks these structural features, 
making it necessary to use a more abstract potential-theoretic framework.

In this section  we assume that $0\in D$ and we let $D_*:=D\setminus\{0\}$.

\begin{definition}
We say that a $(P^V_t)$-excessive   function $h$  is $(P^V_t)$-harmonic provided that for any compact $K\subset V$, we have
\begin{equation}
\label{eq.har}
\mathbb E_x[(\mathbf1_{V}h)(X_{\tau_K})]=h(x),\quad x\in K.
\end{equation}
\end{definition}

\begin{remark}
\label{rem.harmf}
Observe that for any Borel set $B\subset\subset V$ and  $(P^V_t)$-harmonic
function $h$ one has \eqref{eq.har} with $K$ replaced by $B$. Indeed, since $h$
is $(P^V_t)$-excessive, $\mathbb E_x[(\mathbf1_Vh)(X_{\tau_B})]\le h(x),\, x\in V$ (see \cite[Theorem III.5.7]{BG}). 
On the other hand, by $(P^V_t)$-excessiveness of $h$ again  and \eqref{eq.har}
with $K=\overline{B}$, one concludes (see \cite[Theorem III.5.7]{BG} again)
\[
\mathbb E_x[(\mathbf1_Vh)(X_{\tau_B})]\ge \mathbb E_x[(\mathbf1_Vh)(X_{\tau_K})]=h(x),\quad x\in V. 
\]

\end{remark}

\begin{proposition}
\label{prop.harm}
Let $V$ be an open bounded subset of $\mathbb R^d$. 
Suppose that $h\in L^1(V)$ is  positive.
Then $h$  is   $(P^V_t)$-harmonic  if and only if  for any $\eta\in C^2_c(V)$,
\begin{equation}
\label{eq.tharm}
\int_{V}hL^*\eta=0.
\end{equation}
\end{proposition}
\begin{proof}
Observe  that the notion of  $(P^V_t)$-harmonicity  as well as condition \eqref{eq.tharm} do not depend on values of $h$
outside $V$. We redefine $h$ by letting $h=0$ on $V^c$. Then  $(P^V_t)$-harmonicity of $h$ is equivalent to 
\[
\mathbb E_x[h(X_{\tau_K})]=h(x),\quad x\in K
\]
for any compact $K\subset V$. Hence, by translation invariance of L\'evy processes, locally bounded function  $h_\varepsilon :=h\ast j_\varepsilon$
satisfies
\begin{equation}
\label{eq.varhar}
\mathbb E_x[h_\varepsilon(X_{\tau_K})]=h_\varepsilon(x),\quad x\in K
\end{equation}
for any compact $K\subset o_\varepsilon(V)$. By \cite[Theorem 2.18]{MZZ} (see also \cite[Theorem 2.9]{Chen}),
\[
0=\EE(h_\varepsilon,\eta)=\int_{\mathbb R^d}h_\varepsilon L^*\eta=\int_{\mathbb R^d} [j_\varepsilon\ast h] L^*\eta=
\int_{V} h [j_\varepsilon\ast(L^*\eta)]
,\quad \eta\in C^2_c(o_\varepsilon(V)).
\]
This proves the necessity. Suppose now that \eqref{eq.tharm} holds. Let $\varepsilon_0>0$.
Then for any $\varepsilon\in (0,\varepsilon_0)$ and any $\eta\in C_c^2(o_\varepsilon(V))$, we have
\begin{equation}
\label{eq.tharmeps}
0=\int_{V}hL^*(\eta\ast j_\varepsilon)=\int_{\mathbb R^d}(h\ast j_\varepsilon) L^*\eta=\mathcal E(h\ast j_\varepsilon,\eta).
\end{equation}
By \cite{MZZ} $h\ast j_\varepsilon$ is a harmonic function in $o_\varepsilon(V)$, so it satisfies \eqref{eq.varhar}. 
Passing to the limit in \eqref{eq.varhar} gives the sufficiency.  
\end{proof}

\begin{proposition}
\label{prop.punctured-T}
Let $V\subset\mathbb R^d$ be a bounded open set and let $x_0\in V$.
Then
\[
\mathscr T_\nu(V)
=
\mathscr T_\nu(V\setminus\{x_0\}).
\]
\end{proposition}

\begin{proof}
The inclusion $\mathscr T_\nu(V)\subset \mathscr T_\nu(V\setminus\{x_0\})$ is clear.
We now prove the reverse inclusion. Let $u\in\mathscr T_\nu(V\setminus\{x_0\})$.
Fix $K\subset\subset V$. If $x_0\notin K$, then $K\subset V\setminus \{x_0\}$
and $\rho_{K,V}\le \rho_{K,V\setminus\{x_0\}}$, which in turn implies that $u\in L^1_{\rho_{K,V}}(\mathbb R^d)$.
We have to prove that the latter relation also holds under the condition  $x_0\in K$ that we make for the rest of the proof. 
Put $r:=r_{K,V}$ and  $a:= r/10$.
Observe that
\begin{equation}
\label{eq.dfks3}
\begin{split}
\rho_{K,V}(x)&=\nu\bigl(B_r^c\cap(K\cap B^c(x_0,a)-x)\bigr)
+\nu\bigl(B_r^c\cap(x-K\cap B^c(x_0,a))\bigr)\\&\quad+
\nu\bigl(B_r^c\cap(K\cap B(x_0,a)-x)\bigr)
+\nu\bigl(B_r^c\cap(x-K\cap B(x_0,a))\bigr)
\\& \le \rho_{K\cap B^c(x_0,a),V\setminus\{x_0\}}(x)+
\nu\bigl(B_r^c\cap B(x_0-x,a)\bigr)
+\nu\bigl(B_r^c\cap B(x-x_0,a)\bigr),
\quad x\in\mathbb R^d.
\end{split}
\end{equation}
Define
\[
q(x)
:=
\nu\bigl(B_r^c\cap B(x_0-x,a)\bigr)
+\nu\bigl(B_r^c\cap B(x-x_0,a)\bigr),
\]
and write the polar representation of $\nu$,
\[
\nu(E)
=
\int_{\mathbb S^{d-1}}\lambda(d\theta)
\int_0^\infty
\mathbf1_E(t\theta)\frac{dt}{t^{1+2s}}.
\]
An elementary comparison along each ray shows that there exists a constant \(C=C(s)\) such that
for any $x\in\mathbb R^d$ and $\varepsilon\in \{-1,1\}$,
\begin{equation}
\label{eq.ray-comparison}
\begin{split}
\int_0^\infty
\mathbf1_{B_r^c}(t\theta)\mathbf1_{B(\varepsilon x_0,a)}(x+ t\theta)
\frac{dt}{t^{1+2s}}
&=\int_r^\infty\mathbf1_{B(\varepsilon x_0,a)}(x+ t\theta)
\frac{dt}{t^{1+2s}}
\\&\le
C
\int_r^\infty
\mathbf1_{B(\varepsilon x_0,4a)\setminus B(\varepsilon x_0,2a)}(x+t\theta)
\frac{dt}{t^{1+2s}}\\&=C
\int_0^\infty
\mathbf1_{B_r^c}(t\theta)\mathbf1_{B(\varepsilon x_0,4a)\setminus B(\varepsilon x_0,2a)}(x+t\theta)
\frac{dt}{t^{1+2s}}.
\end{split}
\end{equation}
Integrating \eqref{eq.ray-comparison} with respect to the spectral measure 
$\lambda(d\theta)$ yields
\[
q(x)\le C\rho_{B(x_0,4a)\setminus B(x_0,2a),V\setminus\{x_0\}}(x),
\qquad x\in\mathbb R^d.
\]
Together with \eqref{eq.dfks3}, this proves that \(u\in L^1_{\rho_{K,V}}(\mathbb R^d)\). 
Since the compact set \(K\subset V\) was arbitrary, the assertion follows.
\end{proof}

\begin{theorem}
\label{th.main1}
Suppose that  $u\in \mathscr T_\nu(D_*)\cap L^p_{loc}(D_*)$  is  a positive
 solution to \eqref{eq1.0}.
Then 
 $u\in \mathscr T_\nu(D)\cap L^p_{loc}(D)$ and there exists $k\ge 0$  such that
\[
-Lu= u^p+k\delta_{0}\quad \text{in} \quad \mathscr D'(D).
\]
Furthermore, if $p\ge d_s$, then $k=0$ 
\end{theorem}
\begin{proof} 
\medskip 
\noindent 
\textbf{Step 1.} 
In the first step we shall provide  a probabilistic representation   for  $u$.
The representation follows from an abstract result in \cite{K:MAAN} but we shall give a more elementary reasoning.  
For given $\varepsilon,\delta>0$  let 
\[
D^\delta:=o_\delta(D),\quad D^\delta_*:=o_\delta(D)\setminus \{0\}, \quad D^\delta_\varepsilon:=o_\delta(D)\setminus \overline{B_{\varepsilon}}.
\]
By the definition of a  solution to \eqref{eq1.0} 
\begin{equation}
\label{eq1.7}
-\int_{\mathbb R^d}uL^*\eta=\int_{D}u^p\eta,\quad\eta\in C_c^2(D_*).
\end{equation}
Testing \eqref{eq1.7} with $\eta_\varepsilon:=\eta\ast j_\varepsilon$ gives,  for $\varepsilon\le\delta$,
\[
-\int_{\mathbb R^d}uL^*\eta_\varepsilon=\int_{D} u^p\eta_\varepsilon,\quad \eta\in C_c^{\infty}(D_\varepsilon^\delta).
\]
Set
$u^D_\varepsilon:=(\mathbf1_Du)\ast j_\varepsilon\in C_c^\infty(\mathbb R^d)$, and $\nu^{*,D}_{u,\varepsilon}:=\nu^{*,D}_u\ast j_\varepsilon$ (see  \eqref{eq.defnust}). 
The translation invariance of \(L^*\), together with Remark~\ref{rem.imp}, gives
\[
\begin{split}
-\int_{\mathbb R^d}u L^*\eta_\varepsilon&=-\int_{\mathbb R^d}\mathbf1_Du L^*\eta_\varepsilon
-\int_{\mathbb R^d}\mathbf1_{D^c}u L^*\eta_\varepsilon\\&=
-\int_{\mathbb R^d}j_\varepsilon\ast(\mathbf1_Du) L^*\eta-\int_{D}\eta_\varepsilon\,d\nu^{*,D}_u=\int_{\mathbb R^d}Lu^D_\varepsilon\eta-\int_{D}\eta\,d\nu^{*,D}_{u,\varepsilon}.
\end{split}
\]
Consequently,
\[
-\int_{\mathbb R^d}Lu^D_\varepsilon \eta=
\int_{\mathbb R^d}(j_\varepsilon\ast u^p+\nu^{*,D}_{u,\varepsilon})  \eta,\quad \eta\in C_c^{\infty}(D_\varepsilon^\delta).
\]
Fix $\varepsilon_0>0$ and let  $\varepsilon \in (0,\varepsilon_0)$.
Since \(Lu_\varepsilon^D\), \(j_\varepsilon*u^p\), and \(\nu_{u,\varepsilon}^{*,D}\) are continuous, 
we infer that 
\[
-Lu^D_\varepsilon(x)=(j_\varepsilon\ast u^p)(x)+\nu^{*,D}_{u,\varepsilon}(x),\quad x\in D_{\varepsilon_0}^\delta.
\]
Dynkin's formula \cite[Theorem~3.9.4]{Kolokoltsov} therefore yields    
\[
u^D_\varepsilon(x)= \mathbb E_xu^D_\varepsilon(X_{\tau_{D_{\varepsilon_0}^\delta}})+R^{D_{\varepsilon_0}^\delta} (j_\varepsilon\ast u^p)(x)+R^{D_{\varepsilon_0}^\delta}\nu^{*,D}_{u,\varepsilon}(x),\quad x\in D_{\varepsilon_0}^\delta,\, \varepsilon \in (0,\varepsilon_0).
\]
Set $\gamma^\varepsilon_{\delta,\varepsilon_0}(x):=  \mathbb E_xu^D_\varepsilon(X_{\tau_{D_{\varepsilon_0}^\delta}})$.
Since $u$ is non-negative, $\gamma^\varepsilon_{\delta,\varepsilon_0}$
is $(P^{D_{\varepsilon_0}^\delta}_t)$-excessive (see e.g. \cite[Proposition II.2.8]{BG}). 
By Deny's theorem (see Theorem \ref{th-denys}), after passing to a subsequence, \(\gamma_{\delta,\varepsilon_0}^{\varepsilon}\) 
converges almost everywhere to a \((P_t^{D_{\varepsilon_0}^\delta})\)-excessive function \(\gamma_{\delta,\varepsilon_0}\). 
Lemma~\ref{lm2.2} permits passage to the limit in both potential terms, and hence gives  
\begin{equation}
\label{eq4.4}
u=\gamma_{\delta,\varepsilon_0}+R^{D_{\varepsilon_0}^\delta}u^p-R^{D_{\varepsilon_0}^\delta}\nu^{*,D}_{u}\quad\text{a.e. in }D_{\varepsilon_0}^\delta.
\end{equation}

\medskip 
\noindent 
\textbf{Step 2. Passage to the limit as \(\varepsilon_0\searrow0\).} 
By Proposition~\ref{prop.green},
\begin{equation}
\label{eq.conv12}
R^{D_{\varepsilon_0}^\delta}\nu^{*,D}_{u}\to
R^{D^\delta_*}\nu^{*,D}_{u},\qquad R^{D_{\varepsilon_0}^\delta}u^p\to R^{D^\delta_*}u^p
\end{equation}
as $\varepsilon_0\searrow 0$. 
Moreover, Remark~\ref{rem.imp} implies that \(\nu_u^{*,D}(D^\delta)<\infty\) for every \(\delta>0\). 
Together with Lemma~\ref{lm2.2}, this shows that \(R^{D_*^\delta}\nu_u^{*,D}<\infty\) quasi-everywhere in \(D^\delta\) for any $\delta>0$. 
Fatou's lemma applied to \eqref{eq4.4} then yields \(R^{D_*^\delta}u^p<\infty\) almost everywhere. 
Proposition \ref{prop.green} and Remarks~\ref{rem.polar} and \ref{rem.imp} 
imply that  $R^{D_*^\delta}\nu^{*,D}_{u}=R^{D^\delta}\nu^{*,D}_{u}$,
and $R^{D_*^\delta}u^p=R^{D^\delta}u^p$  in $D^\delta$.  It now follows from \eqref{eq.conv12} that \(\gamma_{\delta,\varepsilon_0}\) 
converges as \(\varepsilon_0\searrow0\); we denote its limit by \(\gamma_\delta\).
We thus have
\begin{equation}
\label{eq4.5}
 u(x)=\gamma_{\delta}(x)+R^{D^\delta}u^p(x)+R^{D^\delta}\nu^{*,D}_{u}(x)\quad \text{for a.e. }x\in D^\delta.
\end{equation}

\medskip 
\noindent 
\textbf{Step 3. $\gamma_\delta$ has an excessive version in $D^\delta$.} 
Recall that $\gamma_{\delta,\varepsilon_0}$ is $(P_t^{D_{\varepsilon_0}^\delta})$-excessive, 
which implies (see \cite[Proposition II.2.3]{BG}) that  for any $\alpha>0$
\begin{equation}
\label{eq4.6}
\alpha R^{D_{\varepsilon_0}^\delta}_\alpha (\gamma_{\delta,\varepsilon_0})(x)\le \gamma_{\delta,\varepsilon_0}(x),\quad x\in D_{\varepsilon_0}^\delta.
\end{equation}
Fix $\hat\varepsilon_0\in (\varepsilon_0,\delta)$. Since $D_{\hat \varepsilon_0}^\delta\subset D_{\varepsilon_0}^\delta$, we have
$\tau_{D_{\hat \varepsilon_0}^\delta}\le \tau_{D_{\varepsilon_0}^\delta}$ (see \eqref{eq.hitt}), which combined with \eqref{eq.rdef}, the fact that  $\gamma_{\delta,\varepsilon_0}$
is non-negative, and \eqref{eq4.6} yields
\[
\alpha R^{D_{\hat \varepsilon_0}^\delta}_\alpha (\gamma_{\delta,\varepsilon_0})(x)\le \gamma_{\delta,\varepsilon_0}(x),\quad x\in D_{\varepsilon_0}^\delta.
\]
Letting $\varepsilon_0\searrow 0$ and using Lemma~\ref{lm2.2}, and then letting 
$\hat \varepsilon_0\searrow 0$ and using Proposition  \ref{prop.green}  give
\[
\alpha R^{D^\delta_*}_\alpha\gamma_\delta\le \gamma_\delta,\quad \text{a.e. in } D^\delta_*.
\]
This in turn, by Proposition  \ref{prop.green}, implies that 
\[
\alpha R^{D^\delta}_\alpha\gamma_\delta\le \gamma_\delta,\quad \text{a.e. in } D^\delta.
\]
Consequently, by Theorem~\ref{th-denys} there exists a $(P^{D^\delta}_t)$-excessive function $\hat \gamma_\delta(\cdot)$ such that 
$\hat \gamma_\delta(\cdot)= \gamma_\delta(\cdot)$ a.e. Thus, by \eqref{eq4.5}, 
\begin{equation}
\label{eq.now1}
 u=\hat \gamma_\delta+R^{D^\delta}u^p+R^{D^\delta}\nu^{*,D}_{u},\quad \text{a.e. in }D^\delta.
\end{equation}

\medskip 
\noindent 
\textbf{Step 4. Conclusion of the result.} 
By \cite[Theorem VI.2.11]{BG}
\[
\hat\gamma_\delta=h_\delta+R^{D^\delta}\sigma_\delta\quad\text{in }D^\delta
\]
for some positive Radon  measure $\sigma_\delta$ and positive $(P^{D^\delta}_t)$-harmonic function $h_\delta$.
Consequently
\begin{equation}
\label{eq.delta}
 u=h_\delta+R^{D^\delta}\sigma_\delta+R^{D^\delta}u^p+R^{D^\delta}\nu^{*,D}_{u},\quad \text{a.e. in }D^\delta.
\end{equation}
Fix  $\eta\in C_c^2(D^\delta)$. Multiplying the above equality by (the bounded continuous function) $L^*\eta$ 
and applying Proposition~\ref{prop.harm} and Lemma~\ref{lm.duoper} yields
\[
\begin{split}
-\int_{D^\delta}uL^*\eta&=\int_{D^\delta}h_\delta L^*\eta+\int_{D^\delta} R^{D^\delta}\sigma_\delta L^*\eta+
\int_{D^\delta} R^{D^\delta}u^p L^*\eta+\int_{D^\delta} R^{D^\delta}\nu^{*,D}_{u} L^*\eta
\\&=\int_{D^\delta}\eta\,d\sigma_\delta+\int_{D^\delta}\eta u^p+\int_{D^\delta}\eta\,d\nu^{*,D}_{u}.
\end{split}
\]
Hence
\[
-\int_{D^\delta}uL^*\eta=\int_{D^\delta}\eta\,d\sigma_\delta+\int_{D^\delta}\eta u^p+\int_{D^\delta}\eta\,d\nu^{*,D}_{u},\quad \eta\in C_c^2(D^\delta).
\]
Consequently, there exists a positive Radon measure $\sigma$ on $D$ such that, up to a subsequence, $\sigma_\delta\to \sigma$
in the vague topology (i.e. $\int_{D}\eta\,d\sigma_{\delta}\to \int_D\eta\,d\sigma$ for any $\eta\in C_c(D)$), which in turn 
combined with the previous equality yields
\[
-\int_{D} uL^*\eta=\int_{D}\eta\,d\sigma+\int_{D}\eta u^p+\int_{D}\eta\,d\nu^{*,D}_{u}\quad \eta\in C_c^2(D).
\]
Combining Remark~\ref{rem.imp} with Proposition~\ref{prop.punctured-T}, we obtain 
\[
-\int_{\mathbb R^d}uL^*\eta = \int_D\eta\,d\sigma+\int_D\eta u^p, \qquad \eta\in C_c^2(D).
\]
Comparison with \eqref{eq1.7} shows that \(\operatorname{supp}\sigma\subset\{0\}\), 
which proves the first assertion. Assume now that \(p\geq d_s\). By \eqref{eq.delta}, 
the vague convergence of \((\sigma_\delta)\), and Lemma~\ref{lm2.2}, for every open set \(V\subset\subset D\) 
we have
\[
u\geq R^V\sigma=kG_V(\cdot,0) \quad\text{a.e. in }V. 
\]
If \(k>0\), this inequality would imply 
\(G_V(\cdot,0)\in L^p_{\mathrm{loc}}(V)\), contrary to Corollary~\ref{green-critical}. Hence \(k=0\).
\end{proof}

\begin{corollary}
\label{rem.ineq}
Suppose that $u\in \mathscr T_\nu(D)$ is a positive function  satisfying $-Lu=\mu$ in $\mathscr D'(D)$
for a positive Radon measure $\mu$ on $D$. Then there exists a $(P^D_t)$-harmonic function $h$ such that
\begin{equation}
\label{eq4.4ad}
u=h+R^{D}\mu+R^{D}\nu^{*,D}_{u}\quad\text{a.e. in }D.
\end{equation}
\end{corollary}
\begin{proof}
Repeating the argument in the proof of Theorem~\ref{th.main1} leading
to \eqref{eq.delta}, with $u^p$ replaced by $\mu$, $D_*$ replaced by
$D$, and $\varepsilon=\varepsilon_0=0$, we obtain
\[
u=h_\delta+R^{D^\delta}\sigma_\delta+R^{D^\delta}\mu+R^{D^\delta}\nu^{*,D}_u
\quad\text{a.e. in }D^\delta,
\]
where $\sigma_\delta$ is a positive Radon measure and $h_\delta$ is a
$(P_t^{D^\delta})$-harmonic function. Continuing the argument
following \eqref{eq.delta} and using the identity
\[
-Lu=\mu
\quad\text{in }\mathscr D'(D),
\]
we conclude that $\sigma_\delta\equiv0$. Consequently, the family
$(h_\delta)$ is increasing. The monotone convergence theorem,
Proposition~\ref{prop.green}, and Remark~\ref{rem.imp} therefore yield a
$(P_t^D)$-harmonic function $h$ such that
\[
u=h+R^D\mu+R^D\nu^{*,D}_u\quad\text{a.e. in }D.
\]
\end{proof}

\begin{corollary}
\label{cor.repr194}
If $u$ is a classical solution to \eqref{eq1.1class1}, then $u\in L^p(D)$ and for some $k\ge 0$,
\begin{equation}
\label{eq.repfor23}
u(x)=R^Du^p(x)+kG_D(x,0),\quad x\in D.
\end{equation}
\end{corollary}
\begin{proof}
By Remark \ref{rem.cldist} and Theorem \ref{th.main1} $u$ is a distributional solution to $-Lu=u^p+k\delta_{0}$ in $D$. Thus, by Corollary \ref{rem.ineq}, 
\[
u(x)=h(x)+R^Du^p(x)+kG_D(x,0)+R^{D}\nu^{*,D}_{u}(x).
\]
Since $u=0$ on $D^c$, we have $\nu^{*,D}_{u}\equiv 0$. Furthermore, from  the above equality and  the definition of a $(P^D_t)$-harmonic function
we infer that for any open $V\subset\subset D$
\begin{equation}
\label{eq.harmrepr1}
u(x)=\mathbb E_xu(X_{\tau_V})+R^Du^p(x)+kG_D(x,0)-\mathbb E_x[R^Du^p(X_{\tau_V})]-k\mathbb E_xG_D(X_{\tau_V},0),\quad x\in V.
\end{equation}
Since $u\in C(\mathbb R^d\setminus\{0\})$, we deduce that $\mathbb E_xu(X_{\tau_V})\to 0$ when $V\uparrow D$. On the other hand,
by \cite[Proposition VI.2.10]{BG}, $\mathbb E_x[R^Du^p(X_{\tau_V})]+k\mathbb E_xG_D(X_{\tau_V},0)\to 0$ as $V\uparrow D$ for $x\in V$.
\end{proof}

\begin{corollary}
\label{col.clhol}
If \(u\) is a classical solution of \eqref{eq1.1class1}, then precisely one of the following alternatives holds. 
Either \(k>0\) in \eqref{eq.repfor23}, in which case there exist \(c>1\) and \(\varepsilon>0\) such that 
\[
kG_D(x,0)\leq u(x)\leq ckG_D(x,0), \qquad x\in B_\varepsilon(0),
\]
or \(k=0\), in which case \(u\) is locally H\"older continuous in \(D\).
\end{corollary}
\begin{proof}
 Set \(h:=kG_D(\cdot,0)\).  By Corollary \ref{cor.repr194}
\[
u(x)=T(u)(x)+h(x),\quad x\in D,
\]
where  $T(w):= R^Dw^p$ for  positive $w\in L^p(D)$. 
By Corollary \ref{green-young} there exists $n_0\ge 1$ such that $T^n(u)\in L^\infty(D),\, n\ge n_0$.
If \(k=0\), it follows that \(u\) is bounded, and the interior regularity result \cite[Theorem~1.1]{DipierroRosOtonSerraValdinoci2022} 
yields the local H\"older continuity of \(u\). Suppose now that \(k>0\). 
Repeated substitution in the identity \(u=T(u)+h\) gives
\[
u\le c(n,p)\left [T^n(u)+\sum_{m=0}^{n-1}T^m(h)\right],\quad n\ge 1,
\]  
with the convention $T^0:= Id$. By \cite[(1.13)]{GV} there exists $c>0$ such that $\sum_{m=0}^{n_0}T^m(h)\le ch$ in $D$. 
Since $T^{n_0+1}(u)\in C_b(\mathbb R^d)$ and $G_D(\cdot,0)$ is strictly positive and lower semicontinuous, 
the bounded term is dominated by a constant multiple of \(G_D(\cdot,0)\) in a sufficiently small neighbourhood of the origin. 
The asserted two-sided estimate follows.
\end{proof}

\section{Dirichlet problem for semilinear equations with measure data}
\label{sec7}
For a positive bounded  measure $\mu$ 
on $D$ and a Carath\'eodory function $g:D\times\mathbb R \to\mathbb R^+$ we consider the exterior Dirichlet problem 
\begin{equation}
\label{eq1.111}
\begin{cases}
-Lu=g(\cdot,u)+\mu,\quad &\text{in}\quad D, \\
u=0,\quad &\text{in}\quad\mathbb{R}^d\setminus D,
\end{cases}
\end{equation}

\begin{definition}
We say that $u\in p\mathscr B(\mathbb R^d)$ is a solution to \eqref{eq1.111}
if $g(\cdot,u)\in L^1(D)$ and 
\[
u=R^Dg(\cdot,u)+R^D\mu\quad\text{in }\mathbb R^d.
\]
\end{definition} 

For a positive Borel measurable function
$u$ on $\mathbb R^d$, we   will formulate  the following domination  condition: 
\begin{equation}
\label{zero-boundary-cond}
\text{there exists a positive Borel measure }\,  \mathfrak m \,\text{ on }\,  D\, \text{ such that }\,  u \le R^D\mathfrak m < \infty \,\text{ a.e. in }\mathbb R^d.
\end{equation}
The following lemma clarifies the above condition.

\begin{lemma}\label{equiv-dirichlet}
Let $u\in\mathscr T_\nu(D)$ be  positive and $g(\cdot,u)\in L^1(D)$. 
Suppose that  $-Lu=g(\cdot,u)+\mu$ in $\mathscr D'(D)$. Then
\[
u \text{ is a solution to }\eqref{eq1.111} \iff \text{condition }\eqref{zero-boundary-cond}\text{ is satisfied.} 
\]
\end{lemma}
\begin{proof}
The forward implication follows directly from the integral formulation. 
Conversely, assume that \eqref{zero-boundary-cond} holds and set 
\[
w:=u-R^Dg(\cdot,u)-R^D\mu. 
\]
Remark~\ref{rem.ineq} implies that \(w\geq0\). Moreover, for every \(\eta\in C_c^2(D)\), 
\[ 
-(u,L^*\eta) = (\eta,g(\cdot,u)+\mu) = -(L^*\eta,R^Dg(\cdot,u)+R^D\mu), 
\] and hence 
\[ 
(w,L^*\eta)=0,\quad \eta\in C_c^2(D).
\] 
Proposition~\ref{prop.harm} therefore shows that \(w\) is \((P_t^D)\)-harmonic. 
Since \(w\leq R^D\mathfrak m\), \cite[Theorem~VI.2.11]{BG} implies that \(w=0\). Thus 
\[
u=R^Dg(\cdot,u)+R^D\mu, 
\]
as required.
\end{proof}

\begin{corollary}
\label{col.12}
Suppose that $u=0$ a.e. on $D^c$.
\begin{enumerate}
\item  If $u\in H^s_0(D)$, then $|u|$ satisfies \eqref{zero-boundary-cond}.
\item If $u\in L^\infty(D)$, $s\in (1/2,1)$ and $D$ is  a Lipschitz domain,  then $|u|$ satisfies \eqref{zero-boundary-cond}.
\item If $u\in L^\infty(D)$, $L$ is symmetric  and $D$ is  a Lipschitz domain,  
then $|u|$ satisfies \eqref{zero-boundary-cond}.
\item If $u\in L^1(D)$ is continuous on $\mathbb R^d\setminus o_\delta(D)$ for some $\delta>0$ and $-Lu=\mu$ in $\mathscr D'(D)$
for a positive Radon measure $\mu$ on $D$, then $|u|$ satisfies \eqref{zero-boundary-cond}.
\end{enumerate} 
\end{corollary}
\begin{proof}
The first assertion follows from \cite[Theorem 1.1.1, Theorem 2.3.1]{Oshima}. As to   the assertions (2) and (3)
first observe that by 
\cite[Theorem 1]{Sztonyk2000} and \cite[Pages 228-229]{Millar1975}  $\mathbb P_x(X_{\tau_D-}\in D,\, \tau_D<\infty)=1,\, x\in D$. 
Now, by   \cite[Theorem 3.5.14(ii)]{Oshima} $R^D\kappa_D=1$ on $D$, where $\kappa_D$ is the so-called killing measure 
(it is a positive Radon measure on $D$). Thus, $|u|\le \|u\|_\infty R^D\kappa_D$. 
For (4) it is enough to apply \eqref{eq4.4ad}, in the same manner as in \eqref{eq.harmrepr1}, and then  pass to the limit with $V\nearrow D$.
Note here that since  $u\in L^1(D)$ and $u=0$ a.e. in $D^c$, we have   $u\in L^1(\mathbb R^d)$,
which combined with Remark \ref{rem.incl123} gives that $u\in\mathscr T_\nu(D)$.
\end{proof}

\begin{corollary}
Suppose that $u$ is a positive function on $\mathbb R^d$
such that  $u=0$ a.e. on $D^c$, $u\in L^1(D)$,  $g(\cdot,u)\in L^1(D)$ and  
$-Lu=g(\cdot,u)+\mu$ in $\mathscr D'(D)$ for some positive Radon measure $\mu$ on $D$.    
If one   of the conditions  (1)--(4) in  Corollary~\ref{col.12} is satisfied, then $u$ is a solution to \eqref{eq1.111}.
\end{corollary}

\begin{theorem}\label{exist-sol}
Let $p\in(1,{d_s})$ and let $\mu$ be a positive bounded Borel measure on $D$. Then there exists $\varepsilon_0>0$ such that for any $k \in(0,\varepsilon_0]$ 
problem  \eqref{eq1.1} admits a positive solution.
\end{theorem}
\begin{proof}
Let  $\hat{p}:=(p+{d_s})/2$.
Consider the operator $\mathcal{T}_k:L^{\hat p}(D)\mapsto L^{\hat p}(D)$ 
given by the formula 
\[
\mathcal{T}_k(u):=R^D[|u|^p+k\mu].
\]
By Corollary~\ref{green-young}, Lemma \ref{green-l1} and Lemma~\ref{lm.duoper} the operator  $\mathcal T_k$
is well defined, compact, and  continuous.
Furthermore, by Lemma~\ref{lm.duoper}, we obtain
\begin{equation}
\label{ineq-tn}
\|\mathcal{T}_k(u)\|_{L^{\hat{p}}(D)}\leq C(\|u^p\|_{L^1(D)}+k\mu(D))\leq C(\|u\|^p_{L^{\hat p}(D)}+k\mu(D)),
\end{equation}
where $C$ is a constant that depends on $D,s,d,p$.
Choose $\varepsilon_0>0$ sufficiently small so that there exists a maximal solution $\bar{\lambda}>0$ to the equation
\[
\lambda=C\lambda^p+C\varepsilon_0.
\]
From this point on, fix $k\in(0,\varepsilon_0)$. Consider the set
$\mathcal{G}:=\{u\in L^{\hat p}(D): u\ge0,\, \|u\|_{L^{\hat{p}}(D)}\leq \bar{\lambda}\}$,
which is convex and closed. 
Inequality \eqref{ineq-tn} shows that $\mathcal{T}_k(\mathcal{G})\subset \mathcal{G}$. 
Thus the Schauder fixed point theorem guarantees the existence of $u\in \mathcal{G}$
such that $\mathcal T_k(u)=u$. 
\end{proof}

\begin{theorem}
\label{k-star-sol}
Let $p\in(1,{d_s})$ and $\mu$ be a positive bounded Borel measure on $D$.
Then there exists $0<k_\mu<\infty$ such that for $k\in(0,k_\mu)$ problem  \eqref{eq1.1} 
admits a minimal positive solution, whereas for  $k>k_\mu$ \eqref{eq1.1} admits no positive  solutions.
\end{theorem}
\begin{proof}
Define the sequence of functions as follows
\begin{equation}
\label{iter-leb}
v_0:=kR^D\mu,\quad v_{n+1}:=R^D[v_n^p]+kR^D\mu.
\end{equation}
Observe that $v_1>v_0$ and that for any $n\geq1$ the inequality
\[
v_n\geq v_{n-1}\quad\text{in}\ D
\]
implies that
\[
v_{n+1}=R^D[v_n^p]+kR^D\mu\geq R^D[v_{n-1}^p]+kR^D\mu=v_n.
\]
Thus the sequence $(v_n)$ is increasing in  $n$. 
We will now treat $w\in L^p(D)$ obtained as a solution in Theorem~\ref{exist-sol} with $k:=\varepsilon_0$ 
as a barrier function for small $k$. Note that $w\geq v_0$ and for all $n\geq0$ the inequality $v_n\leq w$ implies that
\[
v_{n+1}=R^D[v_n^p]+kR^D\mu\leq R^D[w^p]+kR^D\mu\leq w,\quad 0<k\leq \varepsilon_0.
\] 
Thus $w$  bounds the sequence $(v_n)$. Now let us define
\[
v(x):=\lim_{n\to\infty}v_n(x),\quad x\in D.
\]
Then $v\in L^p(D)$ and applying the monotone convergence theorem to  
\eqref{iter-leb} yields
\[
v=R^Dv^p+kR^D\mu.
\]
Thus \(v\) is a solution to \eqref{eq1.1}.
It remains to verify its minimality. 
Let \(u\) be any positive solution to \eqref{eq1.1}. 
The definition of the iteration gives \(v_0\leq u\). 
If \(v_n\leq u\), then positivity of the Green operator yields 
\[
v_{n+1} = R^D(v_n^p)+kR^D\mu \leq R^D(u^p)+kR^D\mu = u. 
\]
Induction therefore gives \(v_n\leq u\) for every \(n\geq0\), which, upon letting $n\to\infty$, yields
$v\leq u$.  Hence \(v\) is the minimal solution; from this point onward it will be denoted by \(u_k\). 
Any positive solution corresponding to a parameter \(k'>0\) serves as a barrier for the same iteration whenever 
\(0<k\leq k'\). Consequently, the set of parameters for which \eqref{eq1.1} admits a positive solution is an interval. 
Moreover, the map \(k\mapsto u_k\) is increasing, and therefore so is \(k\mapsto\|u_k\|_{L^p(D)}\).
Thus, let us  denote
\[
k_\mu:=\sup\{k\ge 0: \text{ there exists a positive solution to }\eqref{eq1.1}\}.
\]
It remains to establish that $k_\mu<\infty$. 
Let $(\lambda_1,\varphi_1)$ be the principal eigenpair of $-L^*$ in $H^s_0(D)$, then
\begin{equation}
\label{eq-eigen}
\begin{split}
\lambda_1^{-1}[k(\mu,\varphi_1)+(u_k^p,\varphi_1)]&=k(\mu,R^{*,D}\varphi_1)+(u_k^p,R^{*,D}\varphi_1)\\&
= k(R^D\mu,\varphi_1)+(R^Du_k^p,\varphi_1)=(u_k,\varphi_1).
\end{split}
\end{equation}
The H\"older inequality gives
\[
\int_D u_k\varphi_1\ dx\le \left(\int_D u_k^p\varphi_1\ dx\right)^{\frac{1}{p}}\left(\int_D \varphi_1\ dx\right)^{1-\frac{1}{p}},
\]
which together with \eqref{eq-eigen} implies
\[
\int_D u_k^p\varphi_1\ dx\leq \lambda_1 \left(\int_D u_k^p\varphi_1\ dx\right)^{\frac{1}{p}}\left(\int_D \varphi_1\ dx\right)^{1-\frac{1}{p}},
\]
which gives 
\begin{equation}\label{bound-uk}
\int_D u_k^p\varphi_1\ dx\leq \lambda_1^{\frac{p}{p-1}}\int_D \varphi_1\ dx.
\end{equation}
Combining \eqref{eq-eigen} with \eqref{bound-uk} we obtain
\[
k(\varphi_1,\mu)\leq \lambda_1 \left(\lambda_1^{\frac{p}{p-1}}\int_D \varphi_1\ dx\right)^{\frac{1}{p}}\left(\int_D \varphi_1\ dx\right)^{1-\frac{1}{p}}=\lambda_1^{\frac{p}{p-1}}\int_D \varphi_1\ dx,
\]
which  is equivalent to
\[
k\leq \frac{\lambda_1^{\frac{p}{p-1}}}{(\varphi_1,\mu)}\int_D \varphi_1\ dx
\]
and shows that $k_\mu<\infty$. 
\end{proof}

\begin{remark}
\label{rem.critex}
We can show that at $k_\mu$  we also obtain a solution. 
Indeed,  the mapping $k\mapsto u_k$ is increasing as  $k$ increases.
 Define $u_{k_\mu}:=\lim_{k\to (k_\mu)^-}u_k$. 
The uniform bound in \eqref{bound-uk} implies that $u_{k_\mu}$
 is finite almost everywhere and that 
\begin{equation}\label{uk-star-bound}
\int_D u_{k_\mu}^p\varphi_1\ dx<\infty.
\end{equation}
For every $k\in (0,k_\mu)$
\[
u_k=R^Du_k^p+kR^D\mu.
\]
By the monotone convergence theorem
\[
u_{k_\mu}=R^Du_{k_\mu}^p+k_\mu R^D\mu.
\]
By finiteness of $R^Du_{k_\mu}^p$ a.e. and strict positivity and lower semicontinuity of $G_D$
it follows that $u_{k_\mu}\in L^p_{loc}(D)$.
However, it is difficult to obtain that $u_{k_\mu}\in L^p(D)$. In Section \ref{sec9} we shall prove that 
this can be proved under the additional assumption that
 $L$ is symmetric.
\end{remark}

We conclude this section by proving that, at the threshold, problem
\eqref{eq1.1} admits at most one solution. This result also marks the end
of the part of the paper in which no symmetry assumption on $L$ is
required.

\begin{proposition}
\label{prop.uncrit}
There exists at most one positive solution to \eqref{eq1.1} with $k=k_\mu$.
\end{proposition}
\begin{proof}
Let \(\mathcal T_k\) denote the operator introduced in the proof of Theorem~\ref{exist-sol}, acting on \(L^{\hat p}(D)\), 
and define 
\[
\mathcal K(u)v:=pR^D\bigl(|u|^{p-2}uv\bigr). 
\]
The map \(\mathcal T_k\) is of class \(C^1\), and its Fr\'echet derivative 
at \(u\) is \(\mathcal K(u)\).
Assume, towards a contradiction, that \eqref{eq1.1} admits two positive solutions for \(k=k_\mu\). 
The proof of Theorem~\ref{k-star-sol} provides a minimal solution \(\bar u\). 
Thus, the mean-value representation gives 
\begin{equation}
\label{eq.uni1231}
u_1-\bar u=\mathcal T_{k_\mu}\left(u_1\right)-\mathcal T_{k_\mu}\left(\bar u\right)=\int_0^1 \mathcal K\left(\bar u+t\left(u_1-\bar u\right)\right) d t\left(u_1-\bar u\right)> \mathcal K\left(\bar u\right)\left(u_1-\bar u\right).
\end{equation}
Since both \(\bar u\) and \(G_D\) are strictly positive, the adjoint operator \(\mathcal K^*(\bar u)\) 
is positivity improving and hence ideal irreducible (recall that the closed ideals of $L^p(D)$ are, up to null sets,
of the form $L^p(A)$, where $A$ is a Borel subset of $D$). 
Corollary~\ref{green-young}, Lemma~\ref{lm.duoper}, and Deny's theorem further imply that \(\mathcal K^*(\bar u)\) is compact
(cf.  the reasoning in the proof of Lemma \ref{green-l1}).
By de Pagter's theorem \cite[Theorem~4.2.2]{MeyerNieberg1991}, its spectral radius $r(\mathcal K^*(\bar u))$ is strictly positive. 
The Krein--Rutman theorem \cite[Theorem~4.1.4]{MeyerNieberg1991} therefore provides a strictly positive eigenfunction 
\(\varphi_1\in L^{\hat p'}(D)\) associated with the eigenvalue $\lambda_1=r(\mathcal K^*(\bar u))$. 
Pairing \eqref{eq.uni1231} with \(\varphi_1\) gives 
\[
\langle u_1-\bar u,\varphi_1\rangle > \lambda_1\langle u_1-\bar u,\varphi_1\rangle, 
\]
and hence \(\lambda_1<1\). Define 
\[
\Phi(k,u):=u-\mathcal T_k(u).
\]
Then 
\[
D_u\Phi(k,u)=Id-\mathcal K(u). 
\]
Since the spectral radius of \(\mathcal K(\bar u)\) equals \(\lambda_1<1\), 
the operator \(D_u\Phi(k_\mu,\bar u)\) is an isomorphism. 
Moreover, \(\Phi(k_\mu,\bar u)=0\). 
Consequently, by the implicit function theorem,  there exists $\varepsilon>0$  such that for any $k\in (k_\mu-\varepsilon, k_\mu+\varepsilon)$
there exists a  positive $u_k\in L^{\hat p}(D)$ such that $\Phi(k,u_k)=0$.
This contradicts the definition of \(k_\mu\) and proves uniqueness at the critical parameter.
\end{proof}

\section{Stability and multiplicity of solutions}
\label{sec8}

In this section we assume that $L$ is symmetric, or equivalently, that $\nu(dx)=\nu(-dx)$ and $b\equiv 0$.
In particular, the form $\mathcal E$ is symmetric. We record the following elementary inequality
\begin{equation}
\label{eq.el1}
a^p+pa^{p-1}b\le (a+b)^p\le a^p+p(a+b)^{p-1}b,\quad a,b\ge 0,\, p>1,
\end{equation}
that we shall use frequently in the sequel.

\begin{definition}
We say that the solution $u$ to the equation \eqref{eq1.1} is \em{stable} if
\begin{equation}\label{ineq-stable}
\mathcal E(\xi,\xi)>p\int_D u^{p-1} \xi^2\,dx,\quad\xi\in H^s_0(D),\, \xi\neq 0.
\end{equation}
Moreover we say that the solution $u$ is \em{semi-stable} if the formula \eqref{ineq-stable} holds with non-strict inequality.
\end{definition}

\begin{lemma}
\label{hs-compact}
Let $q>\frac{d}{2s}$. For any function   $a\in L^q(D)$,  define the mapping $\Psi:H^s_0\mapsto\mathbb{R}$ by
\[
\Psi_a(\xi):=\int_D a\xi^2\,dx. 
\]
Then the following holds.
\begin{itemize}
 \item[(i)] if $a_n\to a\in L^q(D)$ and $\xi_n \rightharpoonup \xi$ in $H^s_0(D)$, then $\Psi_{a_n}(\xi_n)\to \Psi_a(\xi)$;
 \item[(ii)] if $(b_n)$ is bounded in $L^q(D)$ and $\xi_n \rightharpoonup \xi$ in $H^s_0(D)$, then $\Psi_{b_n}(\xi-\xi_n)\to 0$.
 \end{itemize}
\end{lemma}
\begin{proof}
From the Sobolev embedding theorem we know that the embedding
\begin{equation}\label{sob-emb-hs}
H^s_0(D)\hookrightarrow L^r(D)
\end{equation}
is compact for all $r\in [1,2d_s)$. Since $q>d_s'$, we have $2q'<2d_s$ so Sobolev embedding \eqref{sob-emb-hs} is compact with $r:=2q'$.
From the H\"older inequality
\[
\begin{split}
|\Psi_a(\xi_1)-\Psi_{\hat a}(\xi_2)|&\leq|\Psi_a(\xi_1)-\Psi_{\hat a}(\xi_1)|+|\Psi_{\hat a}(\xi_1)-\Psi_{\hat a}(\xi_2)|
\\&\le 
\|a-\hat a\|_{L^q(D)} \|\xi^2_1\|_{L^{q'}(D)}+ \|\hat a\|_{L^q(D)} \|\xi^2_1-\xi^2_2\|_{L^{q'}(D)}\\&
\le \|a-\hat a\|_{L^q(D)} \|\xi_1\|^2_{L^{2q'}(D)}\\&\quad+
\|\hat a\|_{L^q(D)}\|\xi_1-\xi_2\|_{L^{2q'}(D)}(\|\xi_1\|_{L^{2q'}(D)}+\|\xi_2\|_{L^{2q'}(D)}).
\end{split}
\]
Assertion (i) now follows from the above inequality and   the Sobolev embedding \eqref{sob-emb-hs}.
(ii) can be proved in a similar way. We have
\[
\begin{split}
|\Psi_{b_n}(\xi_1-\xi_2)|&\le \|b_n\|_{L^q(D)} \|(\xi_1-\xi_2)^2\|_{L^{q'}(D)}=
 \|b_n\|_{L^q(D)} \|\xi_1-\xi_2\|_{L^{2q'}(D)}^{2}.
\end{split}
\]
Now the Sobolev embedding \eqref{sob-emb-hs} completes the proof.
\end{proof}
\begin{theorem}
\label{stability}
Let  $\mu$ be a bounded positive Borel measure on $D$,  $k\in(0,k_\mu)$ and let $u_k$ be the minimal solution to \eqref{eq1.1}. Then $u_k$ is stable.
\end{theorem}
\begin{proof}
Suppose  that $u_k$ is not stable. Set $B_1:=\{\xi\in H^s_0(D): \mathcal E(\xi,\xi)\le 1\}$.
Set
\begin{equation}\label{stability-lambda}
\lambda_1(k):=\sup\left\{p\int_D u_k^{p-1} \xi^2\,dx\ :\ \xi\in B_1 \right\}
=\sup\left\{\frac{p\int_D u_k^{p-1} \xi^2\,dx}{\mathcal E(\xi,\xi)}:\ \xi\in H_0^s(D),\, \xi\neq 0 \right\}.
\end{equation}
Since we supposed that $u_k$ is not stable, we have  $\lambda_1(k)\ge 1$.
Since $u_k\in L^p(D)$, we have $pu_k^{p-1}\in L^{p'}(D)$ and since $p<d_s$, we have 
$p'>d_s'=\frac{d}{2s}$. We can make use of Lemma~\ref{hs-compact} to guarantee that $\lambda_1(k)<\infty$ and that there  exists $\xi_0\in B_1$ such that
\[
p\int_D u_k^{p-1} \xi_0^2\,dx=\lambda_1(k)\ge 1.
\]
We may also assume that $\xi_0$ is nonnegative since $\mathcal{E}(|\xi|,|\xi|)\le\mathcal{E}(\xi,\xi)$ and the formula contains $\xi^2$. 
Consequently, $\xi_0\in H^s_0(D)$  is a global minimizer of the energy functional
\[
E(\xi):=\frac{1}{2}\mathcal{E}(\xi,\xi)-\frac{1}{2}\sigma\int_D pu_k^{p-1} \xi^2\,dx,\quad \xi\in H^s_0(D),
\] 
where $\lambda^{-1}_1(k)=:\sigma\in(0,1]$. Hence $\xi_0$ is a weak solution of the equation
\[
-L\xi_0=\sigma pu_k^{p-1}\xi_0.
\]
Choose $\hat{k}\in(k,k_\mu)$ and let $w_l:=R^D[(u^p_{\hat{k}}-u^p_k)\wedge l]$,
$w_\infty:=R^D[u^p_{\hat{k}}-u^p_k]$, and $v_\infty:= u_{\hat{k}}-u_k$. 
Since $G_D>0$ in $D\times D$, and 
\[
v_\infty=w_\infty+(\hat{k}-k)R^D[\mu],
\]
we conclude that $v_\infty>w_\infty>0$ in $D$.
From  elementary inequality \eqref{eq.el1}
we obtain
\begin{align*}
\sigma\int_D pu_k^{p-1}w_l\xi_0\,dx&=\mathcal{E}(\xi_0,w_l)=\int_D[(u^p_{\hat{k}}-u^p_k)\wedge l]\xi_0\,dx\ge \int_D \left([pu_k^{p-1}(u_{\hat{k}}-u_k)]\wedge l\right) \xi_0\,dx.
\end{align*}
Letting $l\to \infty$ yields
\begin{align*}
\sigma\int_D pu_k^{p-1}w_\infty\xi_0\,dx\ge \int_D pu_k^{p-1}v_\infty\xi_0\,dx>\int_D pu_k^{p-1}w_\infty\xi_0\,dx.
\end{align*}
This leads to a contradiction and completes the proof.
\end{proof}
\begin{corollary}\label{stability-stronger}
Observe that in the proof of Theorem~\ref{stability} we in fact proved a stronger result, i.e.
\begin{equation}\label{stability-strong}
\exists c>0 \quad \mathcal E(\xi,\xi)-p\int_D u^{p-1}_k \xi^2\,dx\ge c\mathcal E(\xi,\xi)\quad \forall\xi\in H^s_0(D).
\end{equation}
This follows from the fact that in the definition of $\lambda_1(k)$ in \eqref{stability-lambda} we obtained a contradiction assuming that $\lambda_1(k)\ge1$, so we have $\lambda_1(k)<1$ and we can take $c:=1-\lambda_1(k)$. We will use this stronger form below.
\end{corollary}
\begin{remark}\label{l2-compact}
We define $L^2(D,a):=\{\xi: D\mapsto \mathbb{R}\ |\ \|\xi\|^2_{L^2(D,a)}:=\int_D\xi^2|a|<\infty\}$.
Observe that for $a\in L^q(D)$ with $q>\frac{d}{2s}$ we have from Lemma~\ref{hs-compact} that 
\[
H^s_0(D)\hookrightarrow L^2(D,a)\quad\text{is compact}.
\]
\end{remark}

\begin{lemma}[Bregman-divergence]
\label{lm.algebraic}
Let $p>1$, $c_p:=\min\{1,p-1\}$, and let
\begin{equation}
\label{eq.efff}
F(s,t):=\frac{1}{p+1}
\Big((s+t^+)^{p+1}-s^{p+1}-(p+1)s^p t^+\Big),
\qquad s,t\ge0.
\end{equation}
Then
\begin{enumerate}
\item[(1)] $F(s,t)\ge 0$, $F_t(s,t)=(s+t)^p-s^p$, and  $F_{tt}(s,t)=p(s+t)^{p-1}$, $s,t\ge 0$; 
\item[(2)]
for every $\varepsilon>0$, there exists $C_\varepsilon>0$ such that
\[
F(s,t)\leq
\frac{p+\varepsilon}{2}s^{p-1}t^2+C_\varepsilon t^{p+1},
\qquad s,t\ge0;
\]
\item[(3)]  for any $s,t\ge 0$,
\[
t^2 F_{t t}(s, t)-\left(1+c_p\right) t F_t(s, t)+c_p t^2 F_{t t}(s, 0)\ge 0;
\]
\item[(4)] for any $s,t\ge 0$,
\[
tF_t(s,t)-(2+c_p)F(s,t)
+
\frac{c_p}{2} t^2 F_{t t}(s, 0)\ge0.
\]
\end{enumerate}
\end{lemma}

\begin{proof}
For (1) note that
\[
F(s,t)=p\int_0^t(t-r)(s+r)^{p-1}\,dr .
\]
From this, and by using, for every $\varepsilon>0$, elementary inequality 
\[
(s+r)^{p-1}
\leq
\left(1+\frac{\varepsilon}{p}\right)s^{p-1}
+C_\varepsilon r^{p-1},
\qquad s,r\ge0,
\]
we obtain (2).
As for (3) observe that the desired inequality is equivalent to 
$$
t^2 F_{t t}(s, t)-tF_t(s, t) \geq c_p\left(tF_t(s, t)-t^2 F_{t t}(s, 0)\right),
$$
and further equivalent to 
\[
f(s)-f(t)+f'(t)(t-s)
\geq
c_p\left(f(t)-f(s)-f'(s)(t-s)\right),\quad 0\le s\le t,
\]
where $f(t)=t^p$.
By using integral representation of convex functions the last inequality  in turn is equivalent to
\[
\int_s^t (r-s)f''(r)\,dr
\geq
c_p \int_s^t (t-r)f''(r)\,dr.
\]
It is an elementary calculation that the function
\[
g(t):=
\int_s^t (r-s)f''(r)\,dr-
c_p\int_s^t (t-r)f''(r)\,dr,\quad t\ge s
\]
satisfies $g(s)=0$, $g'(t)\ge 0,\, t\ge s$, which completes the proof of (3).
For (4) put
\[
\Phi(s,t):=
tF_t(s,t)-(2+c_p)F(s,t)
+\frac{c_pp}{2}s^{p-1}t^2 .
\]
Then $\Phi(s,0)=0$, and $t\Phi_t(s,t)$ equals the left-hand side of inequality in (3).
Hence $t\Phi_t(s,t)\ge0$, and as a result  $\Phi_t(s,t)\ge0$ for
$t>0$. Therefore, since $\Phi(s,0)=0$, we obtain $\Phi(s,t)\ge0$.
\end{proof}

\begin{theorem}
\label{mountain-pass-sol} 
 Let  $\mu$ be a bounded positive Borel measure on $D$ and let $k\in(0,k_\mu)$. 
 Then there exists a second  solution $v_k$ to \eqref{eq1.1} such that 
 $v_k>u_k$, where $u_k$ is the minimal solution.
\end{theorem}
\begin{proof}
First observe that finding the   solution $v_k$ is equivalent to finding a non-trivial weak solution to the following problem
\begin{equation}\label{eq-mountain}
\begin{cases}
-Lw=(u_k+w_+)^p-u_k^p,\quad &\text{in}\quad D, \\
w=0,\quad &\text{in}\quad\mathbb{R}^d\setminus D.
\end{cases}
\end{equation}
By a weak solution to \eqref{eq-mountain} we understand a function  $w\in H_0^s(D)$
such that 
\[
\mathcal E(w,\eta)=\int_D((u_k+w_+)^p-u_k^p)\eta,\quad \eta\in H^s_0(D).
\]
Consider the following energy functional
\[
E(v):=\frac{1}{2}\mathcal{E}(v,v)-\int_D F(u_k,v)\,dx,\quad v\in H^s_0(D),
\]
where $F$ is given by \eqref{eq.efff}. 
Then a critical point of $E$ is a weak solution of \eqref{eq-mountain}. 
By Lemma~\ref{lm.algebraic}(2) for any $v\in H^s_0(D)$ we have
\begin{equation}\label{appr-fst}
\int_D F(u_k,v)\,dx\le \frac{p+\varepsilon}{2}\int_D u_k^{p-1}(v^+)^2\,dx+c_\varepsilon\int_D (v^+)^{p+1}\,dx.
\end{equation}
Since $p+1<2d_s$, the Sobolev theorem gives that the inclusion
\begin{equation}\label{sob-emb-comp}
H^s_0(D)\hookrightarrow L^{p+1}(D)\quad\text{is compact.}
\end{equation}
In particular, combining this with Lemma~\ref{hs-compact} gives that  the energy functional $E$ is well defined on 
$H^s_0(D)$. It is immediate that $E(0)=0$. Moreover, using \eqref{stability-strong}, \eqref{appr-fst} 
and \eqref{sob-emb-comp} we obtain for $v\in H^s_0(D)$ with $\mathcal E(v,v)=1$ and $\varepsilon>0$ sufficiently small that
\begin{align*}
E(tv)&=\frac{1}{2}t^2\mathcal E(v,v)-\int_D F(u_k,tv)\,dx\\
&\geq t^2\left(\frac{1}{2}\mathcal E(v,v)-\frac{p+\varepsilon}{2}\int_D u_k^{p-1}v_+^2\right)-c_\varepsilon t^{p+1}\int_D v_+^{p+1}\,dx\\
&\geq C_1t^2\mathcal E(v,v)-C_2t^{p+1}\mathcal E(v,v)^{(p+1)/2}=C_1t^2-C_2t^{p+1},
\end{align*}
where $C_1,C_2>0$. Let $\sigma>0$ be sufficiently small such that $\beta:=C_1\sigma^2-C_2\sigma^{p+1}>0$. Then
\[
E(v)\ge \beta>0,\quad \text{for }\sqrt{\mathcal E(v,v)}=\sigma.
\]
We next show that there exists $e\in H^s_0(D)$ with $\sqrt{\mathcal E(e,e)}>\sigma$ such that $E(e)\le 0$. 
Let $v_0\in H^s_0(D)$ be nonnegative and satisfy $\mathcal E(v_0,v_0)=1$. Then
\[
F(u_k,v_0)\ge \frac{1}{p+1}v_0^{p+1},
\]
which implies that for $t\ge t_0$,
\[
\begin{split}
E(tv_0)&=\frac{t^2}{2}\mathcal E(v_0,v_0)-\int_D F(u_k,v_0)\,dx\\&
\le \frac{t^2}{2}\mathcal E(v_0,v_0)-\frac{1}{p+1}t^{p+1}\int_D v_0^{p+1}\ dx= \frac{1}{2}t^2-C_3t^{p+1}\le 0,
\end{split}
\]
where $t_0$ is large enough and $C_3:=(p+1)^{-1}\int_D v_0^{p+1}\ dx>0$. Thus, taking $e:=\max\{t_0,\sigma+1\}v_0$, we obtain $\sqrt{\mathcal E(e,e)}>\sigma$ and $E(e)\le0$.
We now show that the functional $E$ satisfies the Palais–Smale compactness condition, i.e. for any sequence $(v_n)\subset H^s_0(D)$ 
satisfying $E(v_n)\to c$ and $E'(v_n)\to0$, 
there exists a convergent subsequence. 
From the fact that $E(v_n)\to c$ we obtain
\[
\frac{1}{2}\mathcal E(v_n,v_n)-\int_D F(u_k,v_n)\,dx=E(v_n)\le C_4,
\]
which, after multiplying both sides by $(2+c_p)$ gives us
\begin{equation}\label{mountain-ineq-functional}
\frac{1}{2}(2+c_p)\mathcal E(v_n,v_n)-(2+c_p)\int_D F(u_k,v_n)\,dx\le C_4(2+c_p).
\end{equation}
From the fact that $E'(v_n)\to 0$ we obtain
\[
\mathcal{E}(v_n,w)-\int_D \left((u_k+v_n^+)^p-u_k^p\right)w\ dx=\langle E'(v_n), w\rangle \le C_5\sqrt{\mathcal E(w,w)},
\]
which, after inserting $-v_n$ in place of $w$ gives us
\begin{equation}\label{mountain-ineq-derivative}
-\mathcal E(v_n,v_n)+\int_D \left((u_k+v_n^+)^p-u_k^p\right)(v_n)_+\ dx\le C_5\sqrt{\mathcal E(v_n,v_n)}.
\end{equation}
Adding up the inequalities \eqref{mountain-ineq-functional} and \eqref{mountain-ineq-derivative} we obtain
\[
\frac{c_p}{2}\mathcal E(v_n,v_n)+\int_D \left[\left((u_k+v_n^+)^p-u_k^p\right)v_n^+-(2+c_p)F(u_k,v_n)\right]dx\le C(2+c_p+\sqrt{\mathcal E(v_n,v_n)}),
\]
with $C:=C_4\vee C_5$.
Applying Lemma~\ref{lm.algebraic}(4) with $s:=u_k$ and $t:=v_n^+$ 
to the integrand of the above formula we obtain
\[
\frac{c_p}{2}\left(\mathcal E(v_n,v_n)-p\int_D u_k^{p-1}v_n^2\ dx\right)\le C(2+c_p+\sqrt{\mathcal E(v_n,v_n)}).
\]
The coercivity estimate \eqref{stability-strong} therefore shows that \((v_n)\) 
is bounded in \(H_0^s(D)\). 
Passing to a subsequence, still denoted by \((v_n)\), there exists \(v\in H_0^s(D)\) such that
\begin{equation}
\label{eq.conv1}
\begin{split}
v_n&\rightharpoonup v\quad\text{in } H^s_0(D), \\
v_n&\to v\quad\text{in }L^{p+1}(D)\ \text{and in } L^2(D,u_k^{p-1}),\\
v_n&\to v\quad\text{almost everywhere}.
\end{split}
\end{equation}
Here we have used the compact embedding \eqref{sob-emb-comp} and Remark~\ref{l2-compact}.
To finish the proof we need to show that $\mathcal E(v_n,v_n)\to\mathcal E(v,v)$. 
Since $v_n\rightharpoonup v$ (hence \((v_n-v)\) is bounded in \(H_0^s(D)\)) and $E'(v_n)\to0$ we obtain
\begin{equation}
\label{eq.conv2}
\mathcal{E}(v_n,v_n-v)-\int_D \left[(u_k+(v_n)_+)^p-u_k^p\right](v_n-v)\,dx=\langle E'(v_n),v_n-v\rangle\to0.
\end{equation}
By H\"older's inequality, we obtain
\begin{align*}
\left|\int_D \left[(u_k+(v_n)_+)^p-u_k^p\right](v_n-v)\,dx\right|&\le C\int_D \left[u_k^{p-1}|v_n|+|v_n|^p\right]|v_n-v|\,dx\\
&\le C\left(\int_D|v_n-v|^2u_k^{p-1}\ dx\right)^{\frac{1}{2}}\left(\int_D|v_n|^2u_k^{p-1}\ dx\right)^{\frac{1}{2}}\\
&+C\left(\int_D|v_n|^{p+1}\ dx\right)^{\frac{p}{p+1}}\left(\int_D|v_n-v|^{p+1}\ dx\right)^{\frac{1}{p+1}}.
\end{align*}
By \eqref{eq.conv1}, the right-hand side of the above inequality tends to zero as $n\to\infty$,
which, in turn, combined with \eqref{eq.conv2} yields 
\[
\mathcal E(v_n,v_n-v)\longrightarrow0. 
\]
Together with the weak convergence of \(v_n\), this implies 
\[
\mathcal E(v_n-v,v_n-v)\longrightarrow0. 
\]
Thus \(v_n\to v\) strongly in \(H_0^s(D)\), and \(E\) satisfies the Palais--Smale condition.
The mountain-pass theorem therefore provides a non-trivial solution \(w\in H_0^s(D)\) to \eqref{eq-mountain}.
\end{proof}

\section{Existence  of the Critical  Solution}\label{critical-parameter}
\label{sec9}

In what follows, we assume that $L$ is symmetric and fix a bounded positive Borel measure $\mu$ on $D$.

The aim of this section is to prove that problem \eqref{eq1.1} admits a   solution at the critical parameter $k=k_\mu$.
To this end, we combine some ideas from \cite{lit:FerreroSaccon2006} with the results obtained in the previous section to derive a uniform $L^p$-bound
for the minimal solutions $u_k$ of \eqref{eq1.1}, independent of $k\in(0,k_\mu)$. Together with Remark \ref{rem.critex}, 
this estimate yields the existence of a solution to \eqref{eq1.1} at the critical parameter $k=k_\mu$.

We consider the iterative sequence   already used in the 
proof of Theorem~\ref{k-star-sol} but we change the notation:
\begin{equation}
\label{def-gamma}
  \gamma_{0,k}:=kR^{D}[\mu],\qquad
  \gamma_{n+1,k}:=R^{D}\!\bigl[\gamma_{n,k}^{p}\bigr]+kR^{D}[\mu],
  \quad n\geq0.
\end{equation}
The sequence is pointwise non-decreasing in $D$ and converges as $n\to\infty$ to the minimal
solution $u_{k}$ of \eqref{eq1.1}.  In particular $0\leq\gamma_{n,k}
\leq u_{k},\, n\ge 1$. For  $n\geq1$  we
introduce the \emph{shifted nonlinearity}
\begin{equation}\label{shift-nonl}
  h_{n,k}(x,s):=
  \begin{cases}
    \bigl(s+\gamma_{n,k}(x)\bigr)^{p}-\gamma_{n,k}(x)^{p}, & s\geq0,\\
    -h_{n,k}(x,-s), & s<0,
  \end{cases}
\end{equation}
and the \emph{remainder term}
\begin{equation}
\label{rem-term}
  f_{n,k}(x):=\gamma_{n,k}(x)^{p}-\gamma_{n-1,k}(x)^{p}.
\end{equation}
Recall the elementary fact that the function 
\[
 (a,b)\mapsto (a+b)^{p}-a^{p},\qquad a,b\geq 0,
\]
is nondecreasing in both variables. Thus
\begin{equation}
\label{eq.mon1}
h_{n,k}(x,t)\le h_{n,k'}(x,t),\,\,\, f_{n,k}(x)\le f_{n,k'}(x)\quad \text{for any }0\le k\le k',\, t\ge 0 \text { a.e. }x\in D.
\end{equation}
For the second inequality, it is enough to observe that 
\[
f_{n,k}=\gamma_{n,k}^{p}-\gamma_{n-1,k}^{p}=(\gamma_{n-1,k}+\eta_{n,k})^{p}-\gamma_{n-1,k}^{p},
\]
where 
\[
\eta_{n,k}:=\gamma_{n,k}-\gamma_{n-1,k},\qquad n\geq 1,
\]
and that $\eta_{n,k}$ is nondecreasing in $k$.
To prove the latter claim we proceed by induction on  $n\ge 1$. For $n=1$, we have $\eta_{1,k}=k^{p}R^{D}[(R^{D}[\mu])^{p}]$, so the claim holds true.
Suppose now that $\eta_{n,k}$ is nondecreasing in $k$.  Then $f_{n,k}$ is
nondecreasing in $k$. Therefore, by the positivity of $R^{D}$,
\[
\eta_{n+1,k}=\gamma_{n+1,k}-\gamma_{n,k}=R^{D}[f_{n,k}]
\]
is also nondecreasing in $k$. This completes the induction.

We associate with \eqref{eq-shifted} the functional
$I_{n,k}\colon H_{0}^{s}(D)\to\mathbb{R}$,
\begin{equation}
\label{func-shifted}
  I_{n,k}(w):=\frac{1}{2}\,\mathcal{E}(w,w)-\int_{D}H_{n,k}(x,w)\,dx
              -\int_{D}f_{n,k}\,w\,dx,
\end{equation}
where $H_{n,k}(x,s):=\int_{0}^{s}h_{n,k}(x,t)\,dt$.  
 
\begin{lemma}\label{lemma-regular}
Let $f_{n,k}$ be the function defined in \eqref{rem-term}.
Then there exists $N\in\mathbb{N}$ such that $f_{N,k}\in L^{q}(D)$ for some $q>d_s'$ and every  $k\ge 0$.
\end{lemma}
 
\begin{proof}
By Lemma \ref{lm.duoper}, $\gamma_{0,k}\in M^{{d_s}}(\mathbb R^d)$, which together with 
Corollary~\ref{green-young} gives $\gamma_{n,k}\in M^{{d_s}}(\mathbb R^d)$ for all $n\ge 1$.
 Note that, by the monotonicity of the sequence $(\gamma_{n,k})_{n\ge 1}$ (see \eqref{def-gamma})
 and elementary inequality \eqref{eq.el1}, we have 
$f_{n,k}\leq p\gamma_{n,k}^{p-1}(\gamma_{n,k}-\gamma_{n-1,k})$. On the other hand, using the monotonicity of the sequence 
$(\gamma_{n,k})_{n\ge 1}$ again, and  \eqref{def-gamma} yields 
 \begin{equation}\label{vn-ineq-boot}
\gamma_{n+1,k}-\gamma_{n,k}\le R^{D}[a\cdot (\gamma_{n,k}-\gamma_{n-1,k})],
 \end{equation}
 where $a:=p\gamma_{n,k}^{p-1}$ is in $L^{p'}(D)$; note that   $p'>d_s'$. 
 Thus, by Corollary  \ref{green-young}(i) it is sufficient to show that $v_{N,k}:=\gamma_{N,k}-\gamma_{N-1,k}$ is bounded for some $N\ge 1$
 independently of $k\ge 0$.
 Let $\delta:=1/d_s'-1/p'>0$ and choose $\beta_0\in(p,{d_s})$  such that the ratio $\frac{1}{\delta\beta_0}\notin\mathbb N$.
 Assuming that $v_{n,k}\in L^{\beta_n}(D)$, with $\beta_n>1$ and  $\frac{1}{\delta\beta_n}\notin\mathbb N$,
 we obtain by using  H\"older's inequality that
 \[
 a\cdot v_{n,k}\in L^{t_n}(D), \quad 1/t_n:=1/p'+1/\beta_n=1/d_s'-\delta+1/\beta_n.
 \]
 If $t_n>d_s'$, then by Corollary~\ref{green-young}(i) and \eqref{vn-ineq-boot}, $v_{n+1,k}$ is bounded and   we are done. 
 Otherwise,  since $\frac{1}{\delta\beta_n}\notin\mathbb N$, we must have $t_n<d_s'$.
 Then, by Corollary~\ref{green-young}(ii) and \eqref{vn-ineq-boot}, we obtain $v_{n+1,k}\in L^{\beta_{n+1}}(D)$ with
 \[
 1/\beta_{n+1}=1/t_n-1/d_s'=1/\beta_n-\delta.
 \]
 It is clear that $\frac{1}{\delta\beta_{n+1}}\notin\mathbb N$.
 Consequently, for any $n\ge 1$ such that $t_n<d_s'$, we have 
 \[
 1/\beta_{n+1}=1/\beta_0-(n+1)\delta,\quad 1/t_{n+1}=1/d_s'-\delta+1/\beta_{n+1}.
 \]
 It follows that   there exists $n_0\ge 1$ such that  $t_{n_0}<d_s'<t_{n_0+1}$.
 By Corollary~\ref{green-young}(i) and \eqref{vn-ineq-boot}, $v_{n_0+2,k}$ is bounded.   Hence, letting $N:=n_0+2$ completes the proof.
\end{proof}

We introduce the shifted function
\begin{equation}
\label{eq.defw}
w_{k}:=u_{k}-\gamma_{N,k}\geq0,
\end{equation}
where $N$ is the natural number from the assertion of Lemma~\ref{lemma-regular}.
Our next goal is to show that $w_k$ is a weak solution to 
\begin{equation}
\label{eq-shifted}
-Lw_{k}=h_{N,k}(x,w_{k})+f_{N,k}\quad\text{in }D,\qquad w_k=0\quad \text{in }D^c.
\end{equation}

\begin{lemma}\label{lemma-H0s}
Let   $k\in(0,k_\mu)$ and let $w_k$ be the function defined by \eqref{eq.defw}.
Then $w_k$ is bounded and $w_{k}\in H_{0}^{s}(D)$. In particular $w_{k}$ is a critical point of $I_{N,k}$ and is a weak solution to \eqref{eq-shifted}:
\begin{equation}
\label{weak-shifted}
    \mathcal{E}(w_{k},\phi)=\int_{D}h_{N,k}(x,w_{k})\,\phi\,dx
                            +\int_{D}f_{N,k}\,\phi\,dx
    \qquad\forall\,\phi\in H_{0}^{s}(D).
  \end{equation}
\end{lemma}
 
\begin{proof}
First, we show that $w_k$ is bounded.
Observe that, by \eqref{eq.el1},
\begin{equation}\label{wk-ineq-boot}
w_k\le pR^{D}[(\gamma_{N,k}+w_k)^{p-1}w_k]+R^{D}[f_{N,k}]
\end{equation}
By Lemma~\ref{lemma-regular} $f_{N,k}\in L^{q}(D)$ with $q>d_s'$, so by Corollary~\ref{green-young} the second term on the right-hand side of \eqref{wk-ineq-boot} is bounded. 
Consequently, taking $a:=p(\gamma_{N,k}+w_k)^{p-1}$ we may use the same {\em bootstrap argument} 
as presented in the proof of Lemma~\ref{lemma-regular} to show that $w_k$ is bounded. 
Recall from Lemma~\ref{trunc-sobolev} that truncated solutions belong to 
$H_{0}^{s}(D)$. Since $w_k$ is bounded, there exists $M>0$ such that $w_k\wedge M=w_k$, which completes the proof.

\end{proof}

As a corollary, we get, by \eqref{eq.mon1}, that 
\begin{equation}
\label{eq.compw1}
w_{k}\le w_{k'}\text{ a.e.},\quad k\le k',\, k,k'\in (0,k_\mu).
\end{equation}
Indeed, it is enough to use $\phi=(w_k-w_{k'})^+$ as a test function in the weak formulation of the equation for $w_k-w_{k'}$.

\begin{lemma}
\label{lemma-coer}
 Let   $k\in(0,k_\mu)$ and let $w_k$ be the function defined by \eqref{eq.defw}.
There exists
  a constant $C_{k}>0$ such that
  \begin{equation}\label{coercivity}
\mathcal{E}(\phi,\phi)-\int_{D}h_{N,k}'(x,w_{k})\,\phi^{2}\,dx=    \bigl\langle I_{N,k}''(w_{k})\phi,\phi\bigr\rangle
    \geq C_{k}\mathcal E(\phi,\phi),
    \qquad\phi\in H_{0}^{s}(D),
  \end{equation}
  where $I_{N,k}$ is the energy functional defined in \eqref{func-shifted}.
  In particular, $w_{k}$ is a strict local minimiser of $I_{N,k}$ in
  $H_{0}^{s}(D)$.
\end{lemma}
 
\begin{proof}
A direct computation, using \eqref{shift-nonl},  gives for $v,\phi,\eta\in H^s_0(D)$,
\begin{equation}\label{secondvar}
\left\langle I_{N,k}''(v)\phi,\eta\right\rangle
=
\mathcal E(\phi,\eta)
-p\int_D h_{N,k}'(x,v)\phi\eta\,dx,
\end{equation}
where $h_{N,k}'(x,s):=\partial_{s}h_{N,k}(x,s)=p(s+\gamma_{N,k}(x))^{p-1}$.
Thus  $I_{N,k}$  is of
class $C^2$ on $H_0^s(D)$. 
Since $u_{k}=w_{k}+\gamma_{N,k}$, we have $h_{N,k}'(x,w_{k})=pu_{k}^{p-1}$.
Therefore from \eqref{secondvar} 
\begin{equation}\label{secondvar2}
  \bigl\langle I_{N,k}''(w_{k})\phi,\phi\bigr\rangle
  =\mathcal E(\phi,\phi)-p\int_{D}u_{k}^{p-1}\phi^{2}\,dx.
\end{equation}
By Corollary~\ref{stability-stronger} there exists $c_{k}>0$ such that
\begin{equation}\label{stab-use}
  \mathcal E(\phi,\phi)-p\int_{D}u_{k}^{p-1}\phi^{2}\,dx\geq c_{k}\mathcal E(\phi,\phi)
  \qquad\forall\,\phi\in H_{0}^{s}(D).
\end{equation}
Combining \eqref{secondvar2} and \eqref{stab-use} yields \eqref{coercivity}
with $C_{k}:=c_{k}$.
 
To see that $w_{k}$ is a strict local minimiser, note that \eqref{coercivity}
implies that $I_{N,k}$ is strictly convex in a neighbourhood of $w_{k}$ in
$H_{0}^{s}(D)$, and $I_{N,k}'(w_{k})=0$ by Lemma~\ref{lemma-H0s}.
\end{proof}

\begin{lemma}\label{uniform_bound}
As $k \uparrow k_\mu$, the sequence $(w_k)$ defined by \eqref{eq.defw} is uniformly bounded in $H_0^s(D)$.
\end{lemma}
\begin{proof}
Testing \eqref{weak-shifted} with $w_k$ gives
\begin{equation}\label{stab-use12}
\mathcal E(w_k,w_k) = \int_D h_{N,k}(x, w_k)w_k dx + \int_D f_{N,k} w_k dx.
\end{equation}
By Lemma~\ref{lemma-coer}  and \eqref{secondvar} we obtain
\begin{equation}\label{stab-use13}
\mathcal E(w_k,w_k) \ge  \int_D h'_{N,k}(x, w_k) w_k^2 dx=\int_D p(w_k + \gamma_{N,k})^{p-1} w_k^2 dx.
\end{equation}
We apply Lemma~\ref{lm.algebraic}(3) with  $s= \gamma_{N,k}$ and $t = w_k$ and get
\[
p(\gamma_{N,k}+w_k)^{p-1}w_k^2 - (1+c_p)[(\gamma_{N,k}+w_k)^{p}-\gamma_{N,k}^p]w_k\ge -c_p p\gamma_{N,k}^{p-1}w_k^2.
\]
Hence
\begin{equation}\label{stab-use14}
\begin{split}
p(\gamma_{N,k}+w_k)^{p-1}w_k^2&\ge (1+c_p) \left[ (\gamma_{N,k}+w_k)^{p}w_k -\gamma_{N,k}^{p}w_k\right]  -  c_pp\gamma_{N,k}^{p-1}w_k^2
\\&=(1+c_p) h_{N,k}(x, w_k)w_k-  c_pp\gamma_{N,k}^{p-1}w_k^2.
\end{split}
\end{equation}
Combining \eqref{stab-use12}--\eqref{stab-use14}, we deduce
\begin{equation}
\label{eq.ee1}
c_p \mathcal E(w_k,w_k)  \le c_p \int_D p\gamma_{N,k}^{p-1} w_k^2 dx + (1+c_p)\int_D f_{N,k} w_k dx.
\end{equation}
Define the Rayleigh quotient associated with the lower-order weight
$$\Lambda_1(k) := \inf_{\phi \in H_0^s(D) \setminus \{0\}} \frac{\mathcal E(\phi,\phi)}{\int_D p\gamma_{N,k}^{p-1} \phi^2 dx}.$$
We want to prove that
\[
\lim_{k\uparrow k_\mu}\Lambda_{1}(k)>1.
\]
Since $\gamma_{N,k}$ is non-decreasing in  $k\in (0,k_\mu)$, it follows that 
$\Lambda_{1}(k)$ is nonincreasing in $k\in (0,k_\mu)$
Therefore, the limit $\lim_{k\uparrow k_\mu}\Lambda_{1}(k)$ exists, 
and it is enough to show that, for a sequence $k_{j}\uparrow k_\mu$, one has
\[
\liminf_{j\rightarrow\infty}\Lambda_{1}(k_{j})>1.
\]
For each $j$, let $\phi_{j}$ be a positive minimizer of $\Lambda_{1}(k_{j})$
normalized to satisfy
\[
\int_D p\gamma_{N,k_j}^{p-1} \phi^2_j\,dx=1.
\]
Then
\[
\Lambda_{1}(k_{j})=\mathcal{E}(\phi_{j},\phi_{j}).
\]
We claim that $(\phi_{j})$ is bounded in $H_{0}^{s}(D)$. Indeed, fix $\overline{k}\in(0,k_\mu)$ 
and choose any $\varphi\in H_{0}^{s}(D)\setminus\{0\}$. For $j$ large enough we have $k_{j}>\overline{k}$ hence
\[
\Lambda_{1}(k_{j})\le\frac{\mathcal{E}(\varphi,\varphi)}{\int_{D}p\gamma_{N,k_{j}}^{p-1}\varphi^{2}dx}
\le\frac{\mathcal{E}(\varphi,\varphi)}{\int_{D}p\gamma_{N,\overline{k}}^{p-1}\varphi^{2}dx}=:C_{\varphi},
\]
and thus
\[
\mathcal{E}(\phi_{j},\phi_{j})=\Lambda_{1}(k_{j})\le C_{\varphi}.
\]
Thus $(\phi_{j})$ is bounded in $H_{0}^{s}(D)$. Passing to a subsequence, we may assume
\begin{align*}
\phi_{j}&\rightarrow\phi^{*} &&\text{weakly in } H_{0}^{s}(D),\\
\phi_{j}&\rightarrow\phi^{*} &&\text{strongly in } L^{r}(D) \text{ for any } r\in[2,2d_s),\\
\phi_{j}(x)&\rightarrow\phi^{*}(x) &&\text{for a.e. } x\in D.
\end{align*}
Thus, by the normalization of $\phi_{j}$ and Lemma~\ref{hs-compact}, we obtain 
\begin{equation}
 \label{equation2}
1=\lim_{j\to \infty}\int_{D}p\gamma_{N,k_{j}}^{p-1}\phi_{j}^{2}dx= \int_{D}p\gamma_{N,k_\mu}^{p-1}(\phi^{*})^{2}dx.
\end{equation}
On the other hand,
\[
\liminf_{j\rightarrow\infty}\Lambda_{1}(k_{j})=\liminf_{j\rightarrow\infty}\mathcal{E}(\phi_{j},\phi_{j})\ge\mathcal{E}(\phi^{*},\phi^{*}).
\]
This together with \eqref{equation2}  yields
\[
\liminf_{j\rightarrow\infty}\Lambda_{1}(k_{j})\ge\frac{\mathcal{E}(\phi^{*},\phi^{*})}{\int_{D}p\gamma_{N,k_\mu}^{p-1}(\phi^{*})^{2}dx}.
\]
Consequently, by \eqref{eq.compw1} and Lemma~\ref{lemma-coer} 
\begin{align*}
\liminf_{j\rightarrow\infty}\Lambda_{1}(k_{j})&\ge\frac{\mathcal{E}(\phi^{*},\phi^{*})}{\int_{D}p\gamma_{N,k_\mu}^{p-1}(\phi^{*})^{2}dx}
>\frac{\mathcal{E}(\phi^{*},\phi^{*})}{\int_{D}p(\gamma_{N,k_\mu}+w_{\overline{k}})^{p-1}(\phi^{*})^{2}dx}\\
&=\lim_{k\uparrow k_\mu}\frac{\mathcal{E}(\phi^{*},\phi^{*})}{\int_{D}p(\gamma_{N,k}+w_{\overline{k}})^{p-1}(\phi^{*})^{2}dx}
\ge\lim_{k\uparrow k_\mu}\frac{\mathcal{E}(\phi^{*},\phi^{*})}{\int_{D}p(\gamma_{N,k}+w_{k})^{p-1}(\phi^{*})^{2}dx}\ge1.
\end{align*}
Returning to the energy estimate \eqref{eq.ee1}, we infer that
\[
c_p \left(1-\frac{1}{\Lambda_1(k)} \right) \mathcal E(w_k,w_k) \le (1+c_p) \int_D f_{N,k} w_k dx \le (1+c_p) \|f_{N,k}\|_q \|w_k\|_{q'}
\]
for any $q>\frac{d}{2s}$. By the Sobolev embedding
\[
c_p \left( 1 - \frac{1}{\Lambda_1(k)} \right) \sqrt{\mathcal E(w_k,w_k)}  \le c(1+c_p) \|f_{N,k}\|_q.
\]
Since $\liminf_{k \uparrow k_\mu} \Lambda_1(k) > 1$ and $f_{N,k}$ is  bounded in $L^q(D)$ for some $q > \frac{d}{2s}$,
see Lemma~\ref{lemma-regular}, this implies that $(w_k)$ is  bounded in $H_0^s(D)$.
\end{proof}

\begin{theorem}
\label{th.critex1}
Let $1 < p < {d_s}$. For the critical parameter $k = k_\mu$, the problem \eqref{eq1.1} admits a minimal  solution.
\end{theorem}
\begin{proof}
By Lemma~\ref{uniform_bound}, the sequence $(w_k)$ is bounded in 
$H_0^s(D)$. Extracting a weakly convergent subsequence, 
$w_k \rightharpoonup w_{k_\mu}$ in $H_0^s(D)$ as $k \uparrow k_\mu$. 
On the other hand $u_k=\gamma_{N,k} + w_{k},\, k\in (0,k_\mu)$.   By Remark \ref{rem.critex}
$u_k\nearrow u_{k_\mu}$ and 
\[
u_{k_\mu}=R^D(u_{k_\mu})^p+k_\mu R^D\mu.
\]
Clearly,  $u_{k_\mu} = \gamma_{N,k_\mu} + w_{k_\mu}$.
To see that $u_{k_\mu} \in L^p(D)$, we analyze both components. First, since $w_{k_\mu} \in H_0^s(D)$, 
the Sobolev embedding ensures $w_{k_\mu} \in L^{p}(D)$. Moreover $\gamma_{N,k_\mu}$ 
belongs to $M^{{d_s}}(\mathbb R^d)$. Since $D$ is bounded and $p < {d_s}$, we have 
$\gamma_{N,k_\mu} \in L^p(D)$. Consequently, the sum $u_{k_\mu} \in L^p(D)$.
\end{proof}

\appendix

\section{Fourier symbol}

By \cite[Theorem 14.10]{Sato} the symbol $\psi$ of the operator $-L$ admits the form
\begin{equation}
\label{eq.symb1}
\psi(\xi)=\int_{\mathbb S^{d-1}}|\xi\cdot z|^{2s} (1-i \tan {\pi s} \operatorname{sgn}(\xi\cdot z))\, \lambda(dz) 
\text { for } s \neq 1/2,
\end{equation}
and
\begin{equation}
\label{eq.symb2}
\psi(\xi)= \int_{\mathbb S^{d-1}}(|\xi\cdot z|+i \frac{2}{\pi}(\xi\cdot z) \log |\xi\cdot z|) \lambda(dz)+i\langle b, \xi\rangle \quad \text { for } s=1/2.
\end{equation}

\begin{remark}
Equations \eqref{eq.symb1} and \eqref{eq.symb2} show 
that the non-degeneracy condition \eqref{ellipticity} is equivalent to 
\begin{equation}
\label{eq.symbolcomp1}
\operatorname{Re}\psi(\xi)\asymp|\xi|^{2s}. 
\end{equation}
They also imply 
\begin{equation}
\label{eq.symbolcomp2}
|\operatorname{Im}\psi(\xi)| \lesssim \operatorname{Re}\psi(\xi). 
\end{equation}
When \(s\neq1/2\), the latter estimate follows directly from \eqref{eq.symb1} and, in fact, 
does not require \eqref{ellipticity}. The case \(s=1/2\) requires a separate argument because the imaginary part 
of the characteristic exponent contains a logarithmic contribution.
Indeed,
\[
\operatorname{Im}\psi(\xi)
=
c\int_{\mathbb S^{d-1}}
(\xi\cdot\theta)\log|\xi\cdot\theta|\,\lambda(d\theta)+\langle b, \xi\rangle.
\]
By  \cite[Theorem 14.10]{Sato}  the centering condition
\[
\int_{\mathbb S^{d-1}}\theta\,\lambda(d\theta)=0
\]
must hold for strictly $1$-stable L\'evy processes.
Writing
\[
\xi=|\xi|\eta,\qquad \eta\in\mathbb S^{d-1},
\]
we obtain
\[
\int_{\mathbb S^{d-1}}
(\xi\cdot\theta)\log|\xi\cdot\theta|\,\lambda(d\theta)
=
|\xi|\log|\xi|\,
\eta\cdot \int_{\mathbb S^{d-1}}\theta\,\lambda(d\theta)
+
|\xi|
\int_{\mathbb S^{d-1}}
(\eta\cdot\theta)\log|\eta\cdot\theta|\,\lambda(d\theta).
\]
By the centering the first term on the right-hand side disappears and we obtain
\[
\operatorname{Im}\psi(\xi)
=
c|\xi|
\int_{\mathbb S^{d-1}}
(\eta\cdot\theta)\log|\eta\cdot\theta|\,\lambda(d\theta)+\langle b, \xi\rangle.
\]
On the other hand, the real part is of the form
\[
\operatorname{Re}\psi(\xi)
=
c'
\int_{\mathbb S^{d-1}}|\xi\cdot\theta|\,\lambda(d\theta)
=
c'|\xi|
\int_{\mathbb S^{d-1}}|\eta\cdot\theta|\,\lambda(d\theta).
\]
Thus, by \eqref{ellipticity},
\[
\operatorname{Re}\psi(\xi)\asymp |\xi|.
\]
Moreover, the function
\[
\eta\mapsto
\int_{\mathbb S^{d-1}}
(\eta\cdot\theta)\log|\eta\cdot\theta|\,\lambda(d\theta)
\]
is bounded on $\mathbb S^{d-1}$, because $t\log|t|\to 0$ as $t\to 0$ and
$\lambda$ is finite. Hence
\[
|\operatorname{Im}\psi(\xi)|
\le C|\xi|,
\]
which in turn combined  with $\operatorname{Re}\psi(\xi)\asymp |\xi|$ gives \eqref{eq.symbolcomp2}.
\end{remark}

\section{Fine topology}

The fine topology (resp. co-fine topology) on $\mathbb R^d$ related to the operator $L$ is the coarsest topology that makes all the $(P_t)$-excessive (resp. $(P^{*}_t)$-excessive) functions continuous. 
By \cite[Proposition 10.8]{Sharpe} a nearly Borel function $u:\mathbb R^d\to [0,\infty]$ is finely continuous (resp. co-finely continuous)
on  $\mathbb R^d$ if and only if   the process 
\begin{equation}
\label{eq.efq}
[0,\infty)\ni t\mapsto u(X_t)\in \overline {\mathbb R }\quad \text{ is right-continuous
at } t=0 \, \mathbb P_x\text{-a.s.},\, x\in\mathbb R^d
\end{equation}
(resp. $\mathbb P^*_x\text{-a.s.},\, x\in\mathbb R^d$).
Let $V$ be a finely open Borel subset of $\mathbb R^d$. Then (see e.g. \cite[Exercise 10.25]{Sharpe})
\begin{equation}
\label{eq.ppp1}
\text{if } u,v\in \mathscr B(\mathbb R^d) \text { are finely continuous on }V \text{ and }u=v \text{ a.e. then } u(x)=v(x),\, x\in V.
\end{equation}
\section{Green functions}

\begin{proposition}
\label{prop.green}
The following holds:
\begin{enumerate}
\item Let $V,U$ be open subsets of $\mathbb R^d$ and $V\subset U$. Then $G_V\le G_U$;
\item Let $(V_n)$ be an increasing  sequence of open subsets of $\mathbb R^d$. Set $U:=\bigcup_{n\ge 1}V_n$.
Then $G_{V_n}\nearrow G_U$;
\item Let $U$ be an open subset of $\mathbb R^d$ and  let $K\subset U$ be  compact and polar:
$\mathbb P_x(\sigma_K<\infty)=0,\, x\in \mathbb R^d\setminus K$.
 Then $G_{U}=G_{U\setminus K}$ on $U\setminus K$.
\end{enumerate}
\end{proposition}
\begin{proof}
(i) First observe that for $x\in V$,
\[
\tau_V\leq \tau_U,\quad \mathbb P_x\text{-a.s.}
\]
Thus, by \eqref{eq.rdef}, for every $f\in p\mathscr B(\mathbb R^d)$  and $x\in V$,
\[
R^V f(x)
=
\mathbb E_x\int_0^{\tau_V} f(X_t)\,dt
\leq
\mathbb E_x\int_0^{\tau_U} f(X_t)\,dt
=
R^U f(x).
\]
By the Green kernel representation \eqref{dual.res2}, we obtain
\[
\int_{\mathbb R^d} G_V(x,y)f(y)\,dy
\leq
\int_{\mathbb R^d} G_U(x,y)f(y)\,dy,
\quad x\in V, f\in p\mathscr B(\mathbb R^d).
\]
Therefore, for every fixed $x\in V$,
\[
G_V(x,y)\leq G_U(x,y)
\quad \text{for a.e. } y\in \mathbb R^d.
\]
By the very definition  the functions $G_V(x,\cdot)$ and
$G_U(x,\cdot)$ are $(P^{*,V}_t)$-excessive and $(P^{*,U}_t)$-excessive, respectively.  
Therefore they are co-finely continuous on $V$. 
Applying \eqref{eq.ppp1} to the dual process, we obtain
\[
G_V(x,y)\leq G_U(x,y),\qquad x,y\in V.
\]
Since both kernels are extended by zero outside the corresponding domains,
this proves (i). 

(ii) By the first part, the sequence $(G_{V_n})$ is increasing and bounded from
above by $G_U$. Hence the pointwise limit
\[
H(x,y):=\lim_{n\to\infty}G_{V_n}(x,y)
\]
exists and satisfies $H\leq G_U$. Furthermore, by \cite[Proposition II.2.2]{BG}, $H(x,\cdot)$ is co-finely 
continuous on $U$ for any $x\in U$.
We claim that $H=G_U$. Since $V_n\uparrow U$, one verifies that 
\[
\tau_{V_n}\uparrow \tau_U,\qquad \mathbb P_x\text{-a.s.},\quad x\in U.
\]
Consequently, for every $f\in p\mathscr B(\mathbb R^d)$ and $x\in U$,
the monotone convergence theorem gives
\[
R^{V_n}f(x)
=
\mathbb E_x\int_0^{\tau_{V_n}} f(X_t)\,dt
\uparrow
\mathbb E_x\int_0^{\tau_U} f(X_t)\,dt
=
R^U f(x).
\]
Equivalently,
\[
\int_{\mathbb R^d}G_{V_n}(x,y)f(y)\,dy
\uparrow
\int_{\mathbb R^d}G_U(x,y)f(y)\,dy.
\]
On the other hand, by monotone convergence,
\[
\int_{\mathbb R^d}G_{V_n}(x,y)f(y)\,dy
\uparrow
\int_{\mathbb R^d}H(x,y)f(y)\,dy.
\]
Therefore
\[
\int_{\mathbb R^d}H(x,y)f(y)\,dy
=
\int_{\mathbb R^d}G_U(x,y)f(y)\,dy,
\qquad f\in p\mathscr B(\mathbb R^d),
\]
and hence, for any $x\in U$,
\[
H(x,y)=G_U(x,y)
\quad \text{for a.e. } y\in \mathbb R^d.
\]
As above, using the co-fine continuity of the Green kernels, or equivalently
the uniqueness of the excessive version of the potential kernel in
\cite[Theorem VI.1.4]{BG}, we conclude that the equality holds pointwise.
Thus
\[
G_{V_n}(x,y)\uparrow G_U(x,y),
\qquad x,y\in U.
\]
Outside $U\times U$ both sides are equal to zero, by our convention. This
proves the second assertion.

(iii) 
Since $K$  is polar, for every $x\in U\setminus K$,
\[
\tau_{U\setminus K}=\tau_U,
\qquad \mathbb P_x\text{-a.s.}
\]
Consequently,  for every $f\in p\mathscr B(\mathbb R^d)$,
\[
R^{U\setminus K}f(x)
=
\mathbb E_x\int_0^{\tau_{U\setminus K}}f(X_t)\,dt
=
\mathbb E_x\int_0^{\tau_U}f(X_t)\,dt
=
R^Uf(x),
\qquad x\in U\setminus K.
\]
By the Green kernel representation, we obtain
\[
\int_{U\setminus K}G_{U\setminus K}(x,y)f(y)\,dy
=
\int_{U\setminus K}G_U(x,y)f(y)\,dy,
\quad f\in p\mathscr B(\mathbb R^d).
\]
Therefore, for every fixed $x\in U\setminus K$,
\[
G_{U\setminus K}(x,y)=G_U(x,y)
\quad \text{for a.e. }y\in U\setminus K.
\]
Since both sides are co-finely continuous in the second variable on
$U\setminus K$, \eqref{eq.ppp1}, applied to the dual process, yields
\[
G_{U\setminus K}(x,y)=G_U(x,y),
\qquad x,y\in U\setminus K.
\]
This proves (iii).
\end{proof}

\begin{proposition}
\label{green-critical-lemma}
For all $r>0$ and $x,y\in D$,
\[
G_{rD}(rx,ry)=r^{2s-d}G_D(x,y).
\]
\end{proposition}

\begin{proof}
Since   the related L\'evy process  is strictly $2s$-stable, it satisfies the scaling property
\[
\left(r^{-1}X_{r^{2s} t}\right)_{t\ge 0}\stackrel{d}{=}(X_t)_{t\ge 0}, \qquad r>0,
\]
i.e. for any $x\in\mathbb R^d$,
\[
\mathbb P_{rx}(r^{-1}X_{r^{2s} \cdot} \in B)= \mathbb P_x(X_\cdot\in B),\quad B\in\mathcal B(\Omega),
\]
where   $(\Omega,\rho)$ is  the Skorokhod space of $\mathbb R^d$-valued c\`adl\`ag functions 
on $[0,\infty)$ (this is a Polish space, see \cite[Section 12]{bil}).
Consequently, 
\[
\mathbb P_{rx}(\tau_{rD}\in dt)=\mathbb P_x(r^{2s}\tau_D\in dt).
\]
From the  definition of the Green operator $G_D$, and the above two relations, we infer that
for any $f\in p\mathscr B(\mathbb R^d)$,
\[
\begin{split}
\int_{rD} f(z) G_{rD}(rx,z)\,dz
&=
\mathbb{E}_{rx} \int_0^{\tau_{rD}} f(X_t)\,dt=r^{2s} \mathbb{E}_x \int_0^{\tau_D} f(rX_u)\,du
\\&=
r^{2s} \int_D f(ry) G_D(x,y)\,dy=r^{2s-d} \int_{rD} f(z) G_D(x,z/r)\,dz.
\end{split}
\]
This implies
\[
G_{rD}(rx,z) = r^{2s-d} G_D(x,z/r)
\]
for almost every $z\in rD$.  Both sides are co-finely continuous in the second variable on
$rD$, so, by \eqref{eq.ppp1} applied to the dual process, we get the result.
\end{proof}

\begin{corollary}
\label{green-critical}
Let $0\in D$. Then the  function $G_D(\cdot,\ 0)$ does not belong to the space $L^{{d_s}}_{loc}(D)$.
\end{corollary}
\begin{proof}
Let $r>0$ be such that $B(0,2r)\subset D$. Then using scaling from Lemma~\ref{green-critical-lemma} and Proposition \ref{prop.green}(i) we obtain, for $x\in B(0,r)$,
\[
G_D(x,0)\ge G_{B(0,2r)}(x,0)=\frac{r^{d-2s}}{|x|^{d-2s}}G_{B(0,2r^2/|x|)}\left(\frac{rx}{|x|},0\right)\ge \frac{r^{d-2s}}{|x|^{d-2s}}G_{B(0,2r)}\left(\frac{rx}{|x|},0\right).
\]
Since $G_{B(0,2r)}$ is lower semicontinuous and non-trivial, we have, by scaling again, that there exists an open 
set $S\subset r\mathbb S^{d-1}$ and $c>0$ such that $G_{B(0,2r)}(x,0)\ge c,\, x\in S$.
Consequently, $G_{B(0,2r)}(x,0)\ge c\mathbf1_A|x|^{2s-d}$, where  $A:=\{x\in B(0,r):  rx/|x|\in S\}$.
It is clear that $\mathbf1_A|x|^{2s-d}\notin L^{{d_s}}(B(0,r))$, which completes the proof.
\end{proof}

\begin{remark}
\label{rem.polar}
Since  $G$ is strictly positive and lower semicontinuous, we conclude, by Proposition~\ref{green-critical-lemma}, that
$G(x,y)$ is unbounded in $x$ for fixed $y\in\mathbb R^d$.  
Therefore, by  comments in the first paragraph following \cite[Proposition VI.4.3]{BG},
any singleton is a polar set.  
\end{remark}

\begin{proposition}
\label{prop.spkg}
For every $r>0$,
\[
G_{B(0,r)}(x,z)>0,
\qquad x,z\in B(0,r).
\]
\end{proposition}

\begin{proof}
Let $x,z\in B(0,r)$, $x\neq z$ and $h>0$ be such that  $B(z,h)\subset\subset B(0,r)\setminus \{x\}$. Set 
$e:=\frac{z-x}{|z-x|}$ and  $K:=[x,z]$.
Set
\[
\rho:=\frac14[(r-|z|)\wedge (r-|x|)\wedge h].
\]
Choose $\varepsilon \in (0,1/2)$ such that $|z-x|\varepsilon<\rho$.
By the strict positivity of $G(\cdot,\cdot)$,
\[
\int_0^\infty
\mathbb{P}_0\bigl(X_t\in B(e,\varepsilon)\bigr)\,dt
=
\int_{B(e,\varepsilon)}G(w,0)\,dw>0.
\]
Consequently, there exists $u_0>0$ such that
\begin{equation}
\label{eq.endpoint-positive}
\mathbb{P}_0\bigl(X_{u_0}\in B(e,\varepsilon)\bigr)>0.
\end{equation}
Since the paths of $X$ are c\`adl\`ag, they are bounded on
$[0,u_0]$ almost surely. Therefore,
there exists $M>0$ such
that
\begin{equation}
\label{eq.basic-event}
p:=
\mathbb{P}_0\left(
X_{u_0}\in B(e,\varepsilon),\
\sup_{0\leq t\leq u_0}|X_t|\leq M
\right)>0.
\end{equation}
Choose $n\in\mathbb{N}$ so large that $\frac{|z-x|M}{n}<\rho$,
and define
\[
\lambda:=\frac{|z-x|}{n},\quad
T:=\lambda^{2s} u_0,
\quad
t_k:=kT,
\quad k=0,\ldots,n.
\]
By strict $2s$-stability,
\[
\bigl(\lambda^{-1}X_{\lambda^{2s} t}\bigr)_{t\geq0}
\overset{d}{=}
(X_t)_{t\geq0}.
\]
Hence \eqref{eq.basic-event} implies
\begin{equation}
\label{eq.scaled-event}
\mathbb{P}_0\left(
X_{T}\in B(\lambda e,\lambda\varepsilon),\
\sup_{0\leq t\leq T}|X_t|\leq\lambda M
\right)=p.
\end{equation}
For $k=1,\ldots,n$, define
\[
A_k:=
\left\{
X_{t_k}-X_{t_{k-1}}
   \in B(\lambda e,\lambda\varepsilon)
\right\}
\cap
\left\{
\sup_{t_{k-1}\leq t\leq t_k}
|X_t-X_{t_{k-1}}|
\leq\lambda M
\right\},
\]
and put
\[
E_k:=\bigcap_{j=1}^{k}A_j,
\qquad E_0:=\Omega.
\]
By independence and stationarity of increments of the L\'evy process
\begin{equation}
\label{eq.iterated-probability}
\mathbb{P}_x(E_k)=p^k,
\qquad k=1,\ldots,n.
\end{equation}
We next verify that on $E_n$ the process stays inside $B(0,r)$ up to
time $t_n$ and reaches $B(z,h)$ at time $t_n$. For $k=1,\ldots,n$ and  $t\in[t_{k-1},t_{k}]$.
\[
X_{t}-X_0
=
X_t-X_{t_{k-1}}+\sum_{j=1}^{k-1}
\bigl(X_{t_j}-X_{t_{j-1}}\bigr).
\]
Therefore,  on $E_n$, using the definition of $A_k$, we obtain, under the measure $\mathbb P_x$,
\[
\begin{split}
\left|
X_t-[x+(k-1)\lambda e]
\right|&=\left|
X_t-X_0-(k-1)\lambda e
\right|
\\&\leq
|X_t-X_{t_{k-1}}|+
\left|
X_{t_{k-1}}-X_0
-(k-1)\lambda e
\right|\\&\leq
\lambda M+(k-1)\lambda\varepsilon<2\rho.
\end{split}
\]
Thus, by the choice of $\rho$ and $\lambda$,
\[
X_t\in B(0,r),
\quad 0\leq t\leq t_n,
\]
and
\[
|X_{t_n}-z|=|X_{t_n}-[x+n\lambda e]|=|X_{t_n}-X_0-n\lambda e| \le  n\lambda\varepsilon
=
|z-x|\varepsilon
<\rho
\]
on $E_n$, which in turn implies that
\[
X_{t_n}\in B(z,h),\quad t_n<\tau_{B(0,r)}
\]
on $E_n$. Consequently, by right continuity of the paths
\[
\mathbb{E}_x
\int_0^{\tau_{B(0,r)}}\mathbf{1}_{B(z,h)}(X_t)\,dt=\int_{B(z,h)} G_{B(0,r)}(x,y)\,dy>0.
\]
Since this holds  for any $z\in B(0,r)\setminus \{x\}$ and $h>0$ such that $B(z,h)\subset\subset B(0,r)\setminus\{x\}$,
we conclude that $G_{B(0,r)}(x,\cdot)$ is strictly positive a.e. in $B(0,r)$. This combined with the fact that $G_{B(0,r)}(x,\cdot)$
is $(P^{*,B(0,r)}_t)$-excessive  yields the desired result.  
\end{proof}

\begin{remark}
\label{rem.any1}
Observe that the assertion of Proposition \ref{prop.spkg} still holds true when $B(0,r)$ is replaced by any bounded 
convex open set $V$. The only change in the proof  lies in the obvious modification of the definition of $\rho$.
\end{remark}

The following criterion is a generalization of \cite[Theorem 1.3]{ChenHuZhao2025}.

\begin{proposition}[A communication criterion for the positivity of the Green function]
\label{prop.com1}
Suppose that $D$ is $\nu$-chain connected: for every
$x,y\in D$ there exist open sets
\[
U_0,\ldots,U_n\subset\subset D
\]
such that $x\in U_0$, $y\in U_n$, $U_n$ is an open ball, and
\[
\nu(U_{k+1}-z)>0,\qquad z\in U_k,\quad k=0,\ldots,n-1 .
\]
Then
\[
G_D(x,y)>0,\qquad x,y\in D .
\]
\end{proposition}
\begin{proof}
Indeed, by the L\'evy system formula (see \cite{IW}), for every $t>0$ and $z\in U_k$,
\[
\begin{split}
\mathbb E_z \sum_{0<s\leq t\wedge \tau_{U_k}} {\bf 1}_{\{X_{s-}\in U_k,\; X_s\in U_{k+1}\}}
&=
\mathbb E_z \int_0^{t\wedge \tau_{U_k}} \nu(U_{k+1}-X_s)\,ds .
\end{split}
\]
On the interval $[0,\tau_{U_k})$ we have
$X_s\in U_k$, and therefore, by the assumptions that we made,
\[
\nu(U_{k+1}-X_s)>0,\quad s\in [0,\tau_{U_k}).
\]
Thus the right-hand side of the equation is strictly positive. Consequently,
\begin{equation}
\label{eq.lsc1}
\mathbb P_z\left(\exists\, s<\tau_{U_k}:\; X_{s-}\in U_k,\ X_s\in {U_{k+1}}\right)>0,\quad z\in U_k.
\end{equation}
For $k=0,\ldots,n-1$
put
\[
\sigma_k=\inf\{t>0:\ X_{t-}\in U_k,\ X_t\in U_{k+1}\}\wedge\tau_D,\qquad \eta_k=\tau_{U_k}.
\]
By \eqref{eq.lsc1}, for every $z\in U_k$,
\[
q_k(z):=\mathbb P_z(\sigma_k\leq \tau_{U_k})>0 .
\]
Define $T_{k+1}:=T_k+\sigma_k\circ\theta_{T_k}$, $T_0:=0$.
Then, by the strong Markov property,
\[
\begin{split}
\mathbb P_x(T_k<\tau_D)&=\mathbb P_x(X_{T_{k-1}}\in U_{k-1},T_k<\tau_D)=\mathbb P_x(X_{T_{k-1}}\in U_{k-1},T_{k-1}+\sigma_{k-1}\circ\theta_{T_{k-1}}<\tau_D)
\\&=\mathbb E_x\left[\mathbf1_{\{X_{T_{k-1}}\in U_{k-1}\}}\mathbb P_{X_{T_{k-1}}}(\sigma_{k-1}<\tau_D)\right]\ge 
\mathbb E_x\left[\mathbf1_{\{T_{k-1}<\tau_D\}}q_{k-1}(X_{T_{k-1}})\right].
\end{split}
\]
Since $q_{k-1}$ is strictly positive on $U_{k-1}$, we conclude that $\mathbb P_x(T_{k-1}<\tau_D)>0$ implies $\mathbb P_x(T_{k}<\tau_D)>0$.
Consequently, since $\mathbb P_x(T_{0}<\tau_D)>0$, we conclude that 
\[
\mathbb P_x(T_{n}<\tau_D)>0.
\]
This is precisely the event that the process reaches the last set $U_n$
along the prescribed chain of jumps  before leaving $D$.
By Proposition~\ref{prop.spkg}
\[
G_{U_n}(z,y)>0,\qquad z\in U_n .
\]
By the strong Markov property, we obtain
\[
\begin{split}
G_D(x,y)&\ge\mathbb E_x\left[{\bf 1}_{\{T_n<\tau_D\}}\, G_{U_n}(X_{T_n},y)\right]>0.
\end{split}
\]
\end{proof}

\subsection*{Acknowledgements}
{\small Kamil Dunst acknowledges support from the Polish National Science Centre under Grant No. 2023/49/N/ST1/03435. 
Tomasz Klimsiak acknowledges support from the Polish National Science Centre under Grant No. 2022/45/B/ST1/01095.}

\end{document}